\documentclass[12pt]{amsart}
\usepackage[utf8]{inputenc} 
\usepackage[active]{srcltx}
\usepackage{a4wide}
\usepackage{amsthm,amsfonts,amsmath,mathrsfs,amssymb}
\usepackage{dsfont}
\usepackage{mathtools}
\usepackage[T1]{fontenc}
\usepackage[utf8]{inputenc}
\usepackage{enumerate}
\usepackage[left=4cm,top=4cm,right=4cm,bottom=4cm]{geometry}
\usepackage{caption}
\usepackage{subcaption}
\numberwithin{equation}{section}
\usepackage[makeroom]{cancel}

\usepackage{geometry}
\usepackage{graphicx}
\usepackage{amssymb}
\usepackage{amsmath}
\usepackage{amsthm}
\usepackage{listings}
\usepackage[colorlinks=true,urlcolor=blue,citecolor=blue,linkcolor=blue]{hyperref}
\usepackage{esint}
\usepackage{xcolor}

\usepackage{cite}
\usepackage{etex}
\usepackage[all]{xy}
\usepackage{pgf,tikz}
\usetikzlibrary{arrows}
\usepackage{pgfplots,pgfcalendar}
\usetikzlibrary{patterns}
\usepackage{hyperref}

\def\D{{\mathbb D}}  \def\T{{\mathbb T}}
\def\C{{\mathbb C}}  \def\N{{\mathbb N}}
\def\Z{{\mathbb Z}}

\def\({\left(}       \def\){\right)}

\newtheorem{theorem}{Theorem}[section]
\newtheorem{lemma}[theorem]{Lemma}
\newtheorem{proposition}[theorem]{Proposition}
\newtheorem{corollary}[theorem]{Corollary}

\theoremstyle{definition}
\newtheorem{definition}[theorem]{Definition}
\newtheorem{example}[theorem]{Example}

\theoremstyle{remark}
\newtheorem{remark}[theorem]{Remark}

\numberwithin{equation}{section}

\DeclareMathOperator*{\esssup}{ess\,sup}

\begin{document}

\title[ Average radial integrability spaces, tent spaces and operators]{Average radial integrability spaces, tent spaces and integration operators}
\author[T. Aguilar-Hern\'andez ]{Tanaus\'u Aguilar-Hern\'andez}
\address{Departamento de Matem\'atica Aplicada II and IMUS, Escuela T\'ecnica Superior de Ingenier\'ia, Universidad de Sevilla,
Camino de los Descubrimientos, s/n 41092, Sevilla, Spain}
\email{taguilar@us.es}

\author[P. Galanopoulos]{Petros Galanopoulos}
\address{Department of Mathematics, Aristotle University of Thessaloniki, 54124, Thessaloniki, Greece}
\email{petrosgala@math.auth.gr }


\date{\today}


\thanks{This research was supported in part by Ministerio de Econom\'{\i}a y Competitividad, Spain, MTM2015-63699-P,  and Junta de Andaluc{\'i}a, FQM-133}

\maketitle

\begin{abstract}
We deal with a Carleson measure type problem  for  the tent  spaces  $AT_{p}^{q}(\alpha)$ in  the unit disc of the complex plane. They consist of the analytic functions of the  tent spaces $T_{p}^{q}(\alpha)$  introduced by Coifman, Meyer and Stein. Well known spaces like the Bergman spaces arise as a special case of this  family.
 Let $s,t,p,q\in (0,\infty)$ and $\alpha >0\,.$
We find necessary and sufficient conditions on a positive Borel measure $\mu$ of the unit disc in order to exist a positive constant $C $ such that
\begin{align*}
\int_{\mathbb T} \left(\int_{\Gamma (\xi)} |f(z)|^{t}\ d\mu(z)\right)^{s/t}\ |d\xi|\leq C \|f\|^s_{T_{p}^{q}(\alpha)} \,,\quad
f\in AT_{p}^{q}(\alpha)\,,
\end{align*}
where $\Gamma (\xi) = \Gamma_M (\xi)=\{ z\in \mathbb D : |1-\bar{\xi} z |< M (1-|z|^2)\},$ $M> 1/2 $ and  $\xi$ is a boundary point of the unit disc.
This problem was originally posed by D. Luecking.
We apply our results to the study of the action of the integration operator $T_g$, also known as Pommerenke operator, between the average integrability spaces $RM(p,q) ,$ for\,$p,q\in [1,\infty)$.
 These spaces have appeared recently in the work of the first author with M.D. Contreras and  L. Rodríguez-Piazza. We also consider the action  from an  $RM(p,q)$ to a Hardy space $H^s$, where $ p,q,s \in [1,\infty)$.
 \end{abstract}

\tableofcontents

\section{Introduction}
Let $\mathbb D$ be the unit disc in the complex plane, $\mathbb T$ be the unit circle and $\mathcal H(\mathbb D)$ be the collection of the analytic functions in $\mathbb D$. 
Assume that  $\mathbb X$  is a Banach space of analytic functions in $\mathbb D$. A positive Borel measure $\mu$ on the unit disc is called $(s,\mathbb X)$-Carleson measure if there is a positive constant $C>0$ such that 
\begin{equation}\label{Carleson}
\int_{\mathbb D}\, |f(z)|^s\,d\mu(z) \leq C \, \|f\|^s_{\mathbb X}\,, \quad f \in \mathbb{X}\,.
\end{equation}
The notion of $(s,\mathbb X)$-Carleson measures appeared naturally in the work of L. Carleson on the theory of interpolating sequences for the Hardy spaces \cite{G}.
Recall that an $f \in \mathcal  H(\mathbb D)$ is in the Hardy space $H^p, \, p\in (0,\infty),$ if 
\begin{equation}\label{Hardy}
\|f\|_{H^p}^p =\sup_{r\in [0,1)}\, \int_0^{2\pi}\, |f(re^{i\theta})|^p\, d\theta < +\infty\,.
\end{equation}
A standard reference for the Hardy space theory is \cite{Du}.

Carleson proved that the $(p,H^p)$-Carleson measures are exactly described by the geometric condition 
\begin{equation}\label{CMHardy}
\sup_{I\subseteq \mathbb T} \frac{\mu(S(I))}{|I|}\,<+\infty,
\end{equation}
where the supremum is taken over all arcs $I$ of the unit circle, $|I|$ is their arc length  and 
$$
S(I)=\left\{ z\in \mathbb D : 1-|I| \leq |z| <1\,\,\, \text{and}\,\,\, \frac{z}{|z|} \in I \right\} 
$$ 
is what we call a Carleson square. 
 Since then, embeddings of the type (\ref{Carleson}) have been studied in the context of several spaces of analytic functions. See \cite{Carleson_1958}, \cite{Carleson_1962}, \cite{Luecking}, \cite{Pelaez_Rattya_Sierra_2015}.
 From this point and after we agree to use the following simplification. If $\mu$ is a positive Borel on $\mathbb D$ that satisfies (\ref{CMHardy})
 then  we will refer to it 
simply as a Carleson measure.

In \cite{Luecking} Luecking, among other Carleson measure type problems, considered a version of (\ref{Carleson})  for measures $\mu$ on the upper half-plane $\mathbb{H}$,   for tent spaces of analytic functions defined on $\mathbb H$ and for the $n$-derivatives of the functions in those spaces.  See Theorem 3,
Section 6 in \cite{Luecking}.
However, the proof that Luecking provided serves for the case of the unit disc as well. 
Since we are interested in the unit disc setting and for $n=0$, which is the case without the derivative, below we state  the question of Luecking  under the
latter considerations.  

We say that a measurable  function $f$  in $\mathbb D$ belongs to $T_{p}^{q}(\alpha)$, $\,p,q,\alpha\in (0,\infty)$,  if 
\begin{equation}\label{def tent}
\|f\|^q_{T_{p}^{q}(\alpha)} = \int_{\mathbb T}\,\left(\int_{\Gamma(\xi)}\,|f(z)|^p\, \frac{dA(z)}{(1-|z|^2)^{2-\alpha}}\right)^{\frac qp}\, |d\xi|<+\infty \,.
\end{equation}
For a  $\xi\in \T$ we denote as 
	\begin{align*}
	\Gamma (\xi) = \Gamma_M (\xi)=\{ z\in \mathbb D : |1-\bar{\xi} z |< M (1-|z|^2)\},
		\end{align*}
		where  $M> 1/2 $. The $dA(z)$ stands for the area Lebesque measure on $\mathbb D$.
The tent spaces $T_{p}^{q}(\alpha)$ were introduced by Coifman, Meyer and Stein\cite{CMS}.   
We are interested in their subspaces 
$$
AT_{p}^{q}(\alpha)= T_{p}^{q}(\alpha) \cap \mathcal H(\mathbb D)\,.
$$
 The restriction $\alpha >0$ has meaning only in the case of the analytic tent spaces since an $AT_{p}^{q}(0)$ does not contain non trivial $f\in \mathcal H(\mathbb D)$.
 It is customary to use in the definition of the tent spaces
the  regions $\Gamma(\xi)$. In the next seccion we will see that they can be equivalently described by using other type of non tangential approach regions.

Luecking dealed with the following question. Consider $s,p,q,\alpha\in (0,\infty)$.  Find the necessary and sufficient conditions on  the positive Borel measure $\mu$ in $\mathbb D$  in order to exist a positive constant $C$ 
such that 
\begin{align}\label{problem1}
\left(\int_{\mathbb{D}} |f(z)|^{s}\ d\mu(z)\right)^{1/s}\leq C \|f\|_{T_{p}^{q}(\alpha)}\,, \quad\quad  f\in AT_{p}^{q}(\alpha)\,.
\end{align}
Taking into account the terminology introduced with (\ref{Carleson}) we are allowed to say that this is 
an  $(s,AT_{p}^{q}(\alpha))-$ Carleson measure problem.

 The  family $AT_{p}^{q}(\alpha)$  includes the Bergman spaces. 	
Observe that in the case $p=q$ and $\alpha=1$, by an application of the Fubini theorem in (\ref{def tent}), we get the Bergman space 
$A^p$ that is the space of those $f\in \mathcal H(\mathbb D)$ with the integrability property
\begin{equation}\label{Bergman}
\int_{\mathbb D}\, |f(z)|^p\, dA(z) <+\infty\,.
\end{equation}
For more information on Bergman spaces we propose to the interested reader \cite{HKZ}, \cite{duren_schuster_2004}.

Luecking in \cite{Luecking} pursued  these ideas even further by posing the following: 

{\bf Problem 1:}\, Let $p,q,s,t, \alpha > 0$. Characterize the positive Borel measures $\mu$ in $\mathbb D$ for which there is a positive constant $C$
such that 
\begin{equation}\label{main problem}
\int_{\T} \left(\int_{\Gamma(\xi)} |f(z)|^{t}\ d\mu(z)\right)^{s/t}\ |d\xi|\leq C \|f\|^s_{T_{p}^{q}(\alpha)} \,,\quad
f\in AT_{p}^{q}(\alpha)\,.
\end{equation}
  
This is a more general problem compaired to (\ref{problem1})  from the point of view that if we set $s=t$ then  a change in the order of integration converts  (\ref{main problem}) to the $(s,T_{p}^{q}(\alpha))-$ Carleson measure problem (\ref{problem1}) for the measure $(1-|z|^2)d\mu(z)$. In addition, if  we set above
 $p=q$ and $\alpha=1$ then Problem 1 comes down to the  work of Z. Wu  \cite{Wu_2006}.

We  answer with the following theorem.
\begin{theorem}\label{main answer}
	Let $0<  p,q,s,t,\alpha <+\infty$, $Z=\{z_k\}$ an $(r,\kappa)$-lattice, and let $\mu$ be a positive Borel measure on $\D$. Then the following are equivalent.
	\begin{enumerate}
		\item There is a constant $C>0$ such that 
		\begin{align*}
\left(\int_{\T} \left(\int_{\Gamma(\xi)} |f(w)|^{t}\ d\mu(w)\right)^{s/t}\ |d\xi|\right)^{1/s}\leq C \|f\|_{T_{p}^{q}(\alpha)},\quad f\in  AT_{p}^{q}(\alpha).
		\end{align*}
		\item The measure $\mu$ satisfies the following:
		\begin{enumerate}
			\item If $0< s<q<+\infty$, $0< t<p<+\infty$, then
			\begin{align*}
			\int_{\T} \left(\sum_{z_k\in \Gamma(\xi)} \left(\frac{\mu^{1/t}(D(z_k,r))}{(1-|z_k|)^{\alpha/p}}\right)^{\frac{pt}{p-t}} \right)^{\frac{(p-t)qs}{(q-s)pt}}\ |d\xi|<+\infty.
			\end{align*}
			\item If $0< s< q<+\infty$, $0< p\leq t<+\infty$, then
			\begin{align*}
			\int_{\T} \left(\sup_{z_k\in \Gamma(\xi)} \frac{\mu^{1/t}(D(z_k,r))}{(1-|z_k|)^{\alpha/p}} \right)^{\frac{qs}{q-s}}\ |d\xi|<+\infty.
			\end{align*}
			\item If $0< q< s<+\infty$, $0< p,t<\infty$ or $0< q=s<+\infty$, $0< p\leq t<+\infty$, then
			\begin{align*}
			\sup_{k} \frac{\mu^{1/t}(D(z_k,r))}{(1-|z_k|)^{\frac{\alpha}{p}+\frac{1}{q}-\frac{1}{s}}}<+\infty.
			\end{align*}
			\item If $0< q=s<+\infty$, $0< t< p<+\infty$, then
			\begin{align*}
			\sup_{\xi\in\T} \left(\sup_{\xi\in I}  \sum_{z_{k}\in T(I)} \left(\frac{\mu^{1/t}(D(z_k,r))}{(1-|z_k|)^{\alpha/p}}\right)^{\frac{pt}{p-t}} (1-|z_k|)\  \right)^{\frac{p-t}{pt}}<+\infty,
			\end{align*}
				where $I$ runs the intervals in $\T$, $S(I)=\left\{z\in\D\ :\ 1-|I|\leq |z|<1\ \text{and}\ \frac{z}{|z|}\in I\right\}$ and $I$ is its arc length.
		\end{enumerate}
	\end{enumerate}
\end{theorem}

\par As a consequence of our main result we present two applications that occure naturally. It turns out that  
this type of  emmbedings 
are the proper tool  for the study of the integration operator 
which was 
introduced by Pommerenke in \cite{Pom}.  To be specific, let 
$$
g: \mathbb D \rightarrow \mathbb C
$$
be an analytic function. We call integration operator the map
$$
T_g(f)(z)=\int_0^z\, f(\zeta)\,g'(\zeta)\, d\zeta\,, \quad\quad z\in \mathbb D,
$$
where $f \in \mathcal H(\mathbb D)$.  
Pommerenke considered $T_g$ on the Hilbert Hardy space $H^2$ of the unit disc. 
He proved that its boundedness  is characterized exactly when the symbol $g$ is a BMOA function.
That work was the starting point for a great number of results about the properties of $T_g$
on several spaces of analytic functions.\,See \cite{Pau_Pelaez_2010}, \cite{Wu_2011}, \cite{Pau_Pelaez_2012}, \cite{MiihkinenPauPeralaWang2020}, \cite{Aguilar-Contreras-Piazza_2}.

A. Aleman and A. Siskakis in \cite{AS} completed the scene for $T_g$  on the Hardy spaces. In 
\cite{AS2} they extended the study to the Bergman spaces where they proved that
\begin{equation}\label{AS}
T_g \in B(A^p),\, \ p\in (0,\infty)\,\,\, \Leftrightarrow\,\,\, g\in \mathcal{B}, 
\end{equation}
where $\mathcal B$ stands for the Bloch space  which consists of those $g \in \mathcal H(\mathbb D)$ 
 with the property
 $$
 \sup_{z \in \mathbb D}\,|g'(z)|\, (1-|z|^2) \, <+\infty\,.
 $$
 Based on this condition, there is defined a more general family of spaces.
 Let $\gamma >0$,  we say that a function $g \in \mathcal H(\mathbb D)$ belongs to the Bloch type space  ${\mathcal B}^{\gamma }$  if
  $$
 \sup_{z \in \mathbb D}\,|g'(z)|\, (1-|z|^2)^{\gamma} \, <\infty\,.
 $$

  Recently, T. Aguilar-Hernández, M. D. Contreras and L. Rodríguez-Piazza in \cite{Aguilar-Contreras-Piazza}, \cite{Aguilar-Contreras-Piazza_2} defined the spaces of  average radial integrability. They say that  a function $f\in \mathcal H(\mathbb D)$ belongs to $RM(p,q),\, p,q\in (0,\infty),$ if 
 \begin{align}\label{RM}
\|f\|^q_{RM(p,q)}=\int_0^{2\pi}\,\left(\int_0^1\, |f(re^{i\theta})|^p\, dr\right)^{\frac qp}\,d\theta<+\infty.
\end{align} 
This expression is the  natural extension of the property of bounded radial  integrability 
	\begin{equation}\label{BRINT}
		\sup_{\theta \in [0,2\pi)} \,\int_0^1\, |f(re^{i\theta})|\, dr < +\infty 
	\end{equation}
	of an $f\in \mathcal H(\mathbb D)$. According to the Féjer-Riesz inequality  the elements of the Hardy space $H^1$ fullfil the property (\ref{BRINT}) \cite{Du}.

For the case $p=\infty$, $0< q\leq  +\infty$, the spaces $RM(\infty,q)$ consist of analytic functions $f$ on $\D$ such that
\begin{align*}
\|f\|^q_{RM(\infty,q)}&=\int_0^{2\pi}\,\left(\sup_{0\leq r<1} |f(re^{i\theta})|\, \right)^{q}\,d\theta<+\infty, \quad \text{ if } q<+\infty,\ \text{and}\\
\|f\|_{RM(\infty,\infty)}&=\|f\|_{H^{\infty}}.
\end{align*}


It is clear that for $p=q$ 
$$ RM(p,q) \equiv A^p \,. $$
Moreover, with the help of the H\"older's inequality we can verify that the $RM(p,q)$
spaces always stand between two Bergman spaces. To be precise, if $p<q$ then
$$
A^q \subset RM(p,q) \subset A^p \,
$$
and if $q<p$ the containments are the other way around. All the inclusions are strict.

Notice that 
$$RM(\infty,q)\equiv H^q\,.$$ 
This is due to the equivalent description of Hardy spaces in terms of the radial maximal function 
that is 
$$
\|f\|^q_{H^q}\asymp \int_0^{2\pi} \sup_{r\in (0,1)}\,|f(re^{i\theta})|^q\, d\theta \,.
$$
See \cite{G}.

The authors  in \cite{Aguilar-Contreras-Piazza_2} proved that the action of 
$T_g$  in $RM(p,q)$ is  similar to that  in Bergman spaces. To put it simply, the Bloch condition characterizes the boundedness of 
$$
T_g : RM(p,q) \rightarrow RM(p,q)\,.
$$
In order to obtain this result, they prove that condition (\ref{RM})  is equivalent to the condition 
$$
\int_0^{2\pi}\,\left(\int_0^1\, |f'(re^{i\theta})|^p (1-r)^p\, dr\right)^{\frac qp}\,d\theta<+\infty\,.
$$

Here we consider the more general question  stated below.

{\bf Problem 2:} Let $p,q,s,t \in [1,\infty) $. Characterize the boundedness of 
$$
T_g : RM(p,q) \rightarrow RM(t,s)\,.
$$

The answer we give is based on an equivalent representation of the $RM(p,q)$  as tent spaces. This is actually an independent result
contained in the structural theory of the Triebel-Lizorkin spaces and \cite{Cohn_Verbitsky_200} can serve as a reference for that. 
Postponing the details for the next section we just present  this equivalent description according to which an $f\in \mathcal H(\mathbb D)$ belongs to an $RM(p,q)$
if and only if 
\begin{equation*}
\int_{\mathbb T}\,\left(\int_{\Gamma(\xi)}\,|f'(z)|^p\, (1-|z|^2)^{p-1}dA(z)\right)^{\frac qp}\, |d\zeta|<+\infty \,.
\end{equation*}

Combining appropriately this condition and the results of Theorem \ref{main answer} we prove the following result.
\begin{theorem}
	Let $1\leq p,q,s,r<+\infty$ and  $g\in \mathcal{H}(\D)$. The following statements are equivalent:
	\begin{enumerate}
		\item The operator $T_{g}:RM(p,q)\rightarrow RM(t,s)$ is bounded.
		\item The function $g\in \mathcal{H}(\D)$ satisfies that
	\begin{enumerate}
\item If $1\leq s<q<+\infty$, $1\leq t<p<+\infty$, 
\begin{align*}
g\in RM\left(\frac{pt}{p-t},\frac{qs}{q-s}\right).
\end{align*}
\item If $1\leq s< q<+\infty$, $1\leq p\leq t<+\infty$, 
\begin{align*}
\int_{\mathbb T} \left(  \esssup_{z\in \Gamma(\xi)} |g'(z)|(1-|z|)^{1+\frac{1}{t}-\frac{1}{p}}  \right)^{\frac{qs}{q-s}} \,|d\xi|<+\infty
\end{align*}
\item If $1\leq q< s<+\infty$, $1\leq p,t<\infty$ or $1\leq q=s<+\infty$, $1\leq p\leq t<+\infty$, 
\begin{align*}
g\in\mathcal{B}^{1+\frac{1}{t}+\frac{1}{s}-\frac{1}{p}-\frac{1}{q}}.
\end{align*}
\item If $1\leq q=s<+\infty$, $1\leq t< p<+\infty$, the measure $$d\mu(z)=|g'(z)|^{\frac{pt}{p-t}}(1-|z|)^{\frac{pt}{p-t}}\ dm(z)$$ is a Carleson measure.
	\end{enumerate}
\end{enumerate}
\end{theorem}

The $RM(p,q)$  are not the first example of spaces that can be treated as tent. The Littlewood-Paley theory 
provides an alternative way to define the Hardy spaces in terms of the Lusin area function \cite{Z}. An $f\in H^p$ if and only if 
\begin{equation}\label{Hardy2}
\int_{\mathbb T}\,\left(\int_{\Gamma(\xi)}\,|f'(z)|^2\, dA(z)\right)^{\frac p2}\, |d\xi|<+\infty \,.
\end{equation}
It has been proved that  (\ref{Hardy2}) not only is  comparable to  (\ref{Hardy}) but also it is convenient for the study of integration operator.  
Z. Wu, using (\ref{Hardy2}), studied the $T_g$ operator from a  Bergman space to a Hardy for the unit disc \cite{Wu_2011}.
Later, S. Miihkinen, J. Pau,  A. Perälä and M. Wang completed and extended the results of Wu for the case of the unit ball \cite{MiihkinenPauPeralaWang2020}.
Taking into account  the fact that each average radial integrability space is between two Bergman spaces is tempting to wonder
about the action of  $T_g$ from an $RM(p,q)$ to a $H^s$. We prove the following theorem.

\begin{theorem}
	Let $1\leq p,q,s<+\infty$. The following statements are equivalent:
	\begin{enumerate}
		\item The operator $T_g: RM(p,q)\rightarrow H^s$ is bounded.
				\item The function $g\in \mathcal{H}(\D)$ satisfies that
		\begin{enumerate}
			\item If $1\leq s<q<+\infty$, $2<p<+\infty$, 
			\begin{align*}
					|g'(z)|(1-|z|)^{\frac{1}{2}}\in T_{\frac{2p}{p-2}}^{\frac{qs}{q-s}}(1).
			\end{align*}
			\item If $1\leq s< q<+\infty$, $1\leq p\leq 2$, 
			\begin{align*}
	\int_{\mathbb T} \left(  \esssup_{z\in \Gamma(\xi)} |g'(z)|(1-|z|)^{1-\frac{1}{p}}  \right)^{\frac{qs}{q-s}} \,|d\xi|<+\infty\,.
						\end{align*}
			\item If $1\leq q< s<+\infty$, $1\leq p<\infty$ or $1\leq q=s<+\infty$, $1\leq p\leq 2$, $g\in\mathcal{B}^{1+\frac{1}{s}-\frac{1}{p}-\frac{1}{q}}$.
			\item If $1\leq q=s<+\infty$, $2< p<+\infty$, the measure $|g'(z)|^{\frac{2p}{p-t}}(1-|z|)^{\frac{p}{p-t}}\ dm(z)$ is a Carleson measure.
		\end{enumerate}
 \end{enumerate}	
\end{theorem}

Closing the section we present the structure of the paper. In Section 2 we present the preliminary information that will be required throughout our study. We gather the definition of the spaces of analytic functions of our interest and some properties of them that we will need. Moreover, in Section 3,  we give  an identification of the $RM(p,q)$ spaces as tent spaces $AT_{p}^{q}(1)$ (see Proposition~\ref{non_tangential_char_without_der}).

The main results of this article appear in Sections 4 and 5. In Section 4, first we provide a characterization of the positive Borel measures such that they satisfy \eqref{main problem} (see Theorem~\ref{theoremCarlesonTent}). This result will be the cornerstone in the course of the proof of the last section. Second, we focus on Theorem 3 of \cite{Luecking}. Although  we get  the impression that all the cases of the indices involved are covered  it appears that there is a specific case to be made clear. Following  similar  ideas to those employed by  Luecking we complete the picture.

In Section 5 we provide a characterization of when the integration operator $T_g: RM(p,q)\rightarrow RM(t,s)$ is bounded, where $1\leq p,q,t,s<+\infty$ (see Theorem~\ref{RMRM}). Also, we show the boundedness of $T_{g}$ for $t=\infty$, that is, when the operator $T_{g}$ maps $RM(p,q)$ into the Hardy spaces $H^s$ (see Theorem~\ref{RMH}).

Throughout the paper the letter $C = C(\cdot)$ will denote an absolute constant whose
value depends on the parameters indicated in the parenthesis, and may change from
one occurrence to another. We will use the notation $a\lesssim b$ if there exists a constant
$C = C(\cdot) > 0$ such that $a \leq Cb$, and $a\gtrsim b$ is understood in an analogous manner. In
particular, if $a \lesssim b$ and $a \gtrsim b$, then we will write $a \asymp b$.

\section{Definitions and first properties}
\addtocontents{toc}{\protect\setcounter{tocdepth}{1}}

In this section we recall the definitions of the spaces of analytic functions under discussion. Moreover, we compile some properties for the sake of being self-contained.

\subsection{Average radial integrability spaces}
These are the $RM(p,q)$ spaces introduced in \cite{Aguilar-Contreras-Piazza}.
\begin{definition}
	Let $0<p,q\leq +\infty$. We define the spaces of analytic functions
	$$RM(p,q)=\{f\in \mathcal{H}(\D)\ :\ \rho_{p,q}(f)<+\infty\},$$
	where
	\begin{equation*}
	\begin{split}
	\rho_ {p,q}(f)&=\left(\frac{1}{2\pi}\int_{0}^{2\pi} \left(\int_{0}^{1} |f(r e^{i t})|^p \ dr \right)^{q/p}dt\right)^{1/q}, \quad \text{ if } p,q<+\infty,\\
	\rho_ {p,\infty}(f)&=\esssup_{t\in[0,2\pi)}\left(\int_{0}^{1} |f(r e^{i t})|^p \ dr \right)^{1/p}, \quad \text{ if } p<+\infty, \\
	\rho_ {\infty,q}(f)&=\left(\frac{1}{2\pi}\int_{0}^{2\pi} \left(\sup_{r\in [0,1)} |f(r e^{i t})| \right)^{q}dt\right)^{1/q},\quad\text{ if } q<+\infty,\\
	\rho_{\infty,\infty}(f)&=\|f\|_{H^{\infty}}=\sup_{z\in\D}|f(z)|.
	\end{split}
	\end{equation*}
\end{definition}

\subsection{Tent spaces }
The work of Coifman, Meyer and Stein  is considered the starting point for the study of the tent spaces of measurable functions \cite{CMS}. 
Since then they have been studied widely by many authors.  A typical collection of references on the subject is the following  \cite{Luecking1987}, \cite{Luecking},   \cite{Jevtic_1996}, \cite{Arsenovic}, \cite{Cohn_Verbitsky_200}, \cite{Perala_2018}, \cite{MiihkinenPauPeralaWang2020}.

Let $\xi\in \T$. We define the non-tangential regions $\Gamma(\xi)$ and $\Lambda (\xi)$ as follows
\begin{align*}
\Gamma (\xi) & = \Gamma_M (\xi) :=\left\{z\in\D: |z-\xi|< M (1-|z|^2)  \right\},\\
\Lambda(\xi)&:=\left\{r\xi e^{i\theta}: |\theta|< (1-r)\right\},
\end{align*}
where $ M>\frac{1}{2}$ and $|\theta|:=\min\left\{|\theta+2k\pi|:k\in\Z \right\}$. 

\begin{definition}
	Let $0<p,q, \alpha<+\infty$. The tent space  $T_p^q(\alpha)$ consists of measurable functions $f$ on $\D$ such that 
	\begin{align}\label{TentA}
	\|f\|_{T_{p}^{q}(\alpha)}=\left(\int_{\T} \left(\int_{\Gamma(\xi)} |f(z)|^p \ \frac{dA(z)}{(1-|z|^2)^{2-\alpha}} \right)^{q/p}\ |d\xi|\right)^{1/q}<+\infty.
	\end{align}
	Analogously, the space $T_{\infty}^{q}(\alpha)$ consists of measurable functions $f$ on $\D$ with
	\begin{align}\label{TentB}
	\|f\|_{T_{\infty}^{q}(\alpha)}=\left(\int_{\T} \esssup_{z\in \Gamma(\xi)} |f(z)|^{q}\  |d\xi|\right)^{1/q},\quad\text{ if } q<+\infty,
	\end{align}
	where the essential supremum is taken with respect to the Lebesgue measure. Notice that the definition of $T_{\infty}^{q}(\alpha)=T_{\infty}^{q}$ is independent of $\alpha$.
	
	For the case $q=+\infty$ and $p<+\infty$, the space $T_{p}^{\infty}(\alpha)$ consists of measurable functions $f$ on $\D$ with
	\begin{align}\label{TentC}
	\|f\|_{T_p^\infty(\alpha)}=\sup_{\xi\in \T} \left(\sup_{w\in\Gamma(\xi)} \frac{1}{(1-|w|^2)} \int_{S(w)} |f(z)|^{p}\ \ \frac{dA(z)}{(1-|z|^2)^{1-\alpha}} \right)^{1/p}<+\infty,
	\end{align}
	where 
	$$S(re^{i\theta})=\left\{\rho e^{it}\ : 1-\rho\leq 1-r,\ |t-\theta|\leq \frac{1-r}{2}\right\}$$ 
	for $re^{i\theta}\in \D\setminus\{0\}$ and $S(0)=\D$. Notice that $f\in T_\infty^q(\alpha)$ if and only if the measure $d\mu_f(z)=|f(z)|^p (1-|z|^2)^{\alpha-1}\ dA(z)$ is a Carleson measure on $\D$.
	
	If we take holomorphic functions instead of measurable functions, we will denote the tent space of holomorphic functions as $AT_p^{q}(\alpha):=T_p^q(\alpha)\cap \mathcal{H}(\D)$.
\end{definition}

Below we state a technical lemma, well known to the experts of the area. It will be usefull for us too.
\begin{lemma}{\cite[Lemma 4, p. 66]{Arsenovic}}\label{estimate 1} 
	Let $0<p,q<+\infty$ and $\lambda>\max\{1,p/q\}$. Then, there are constants $C_{1}=C_{1}(p,q,\lambda,C)$ and $C_{2}=C_{2}(p,q,\lambda,C)$ such that
	\begin{align*}
	C_1\int_{\T} \mu(\Gamma (\xi))^{q/p} |d\xi|\leq\int_{\T}\left(\int_{\D} \left(\frac{1-|z|}{|1-z\overline{\xi}|}\right)^{\lambda} \ d\mu(z)\right)^{q/p}|d\xi|\leq C_{2} \int_{\T} \mu (\Gamma(\xi))^{q/p} |d\xi|,
	\end{align*}
	for every positive measure $\mu$ on $\D$. 
\end{lemma}
\begin{remark}\label{independetregion}
	By Lemma~\ref{estimate 1}, it is easy to see that  any of the non-tangential regions $\Gamma(\xi)$ in \eqref{TentA} can be replaced by
	 $\Lambda(\xi)$. 
\end{remark}

Moreover, for the case $p=+\infty$ one obtains by using \cite[Theorem 17.11(a), p. 340]{rudin_real_1987} that the definition of $T_{\infty}^{q}$ is independent of the nontangential region.
\begin{remark}\label{independentTinftyq}
	One can obtain an equivalent norm in $AT_{\infty}^{q}$ replacing in \eqref{TentB} the set $\Gamma(\xi)$ by any other Stolz region $\Gamma_{M}(\xi)$.
\end{remark}

\subsection{Tent spaces of sequences }\label{Section2.3} 
Let $\beta(z,w)$ be the hyperbolic metric on $\D$ and let $D(z,r)=\{w\in\D\ :\ \beta(z,w)<r\}$ be the hyperbolic disc of radius $r>0$ centered at $z\in\D$. The sequence $Z=\{z_k\}$ is a separated sequence if there is a constant $\delta>0$ such that $\beta(z_k,z_j)\geq \delta$ for $j\neq k$. The sequence $Z=\{z_k\}$ is said to be an $(r,\kappa)$-lattice (in the hyperbolic distance), for $r>\kappa>0$, if  
\begin{enumerate}
	\item $\D=\bigcup_{k} D(z_k,r)$,
	\item the sets $D(z_k,\kappa)$ are pairwise disjoint,
\end{enumerate}
Notice that any $(r,\kappa)$-lattice is a separated sequence. 

\begin{remark}\label{remarkestimatedistbound}	
	It is known that if $z,w\in\D$ such that $\beta(z,w)<r$, then there is a constant $C=C(r)>0$ such that
	$\frac{1}{C}(1-|z|)\leq 1-|w|\leq C(1-|z|)$. 
\end{remark}

Using arguments similar to those in \cite{duren_schuster_2004}, one can prove the following result. 

\begin{proposition}\label{Ncoverdiscs}Let $K\geq 1$, $R> 0$, and an $(r,\kappa)$-lattice $Z=\{z_k\}$. There is a positive integer $N=N(K,R,Z)$ such that for each point $z\in\D$ there are at most $N$ hyperbolic discs $D(z_k,Kr)$ satisfying $D(z,R)\cap D(z_{k},Kr)\neq\emptyset$.
\end{proposition}
%

As a consequence of the above proposition we obtain the following covering lemma which will be useful in Section 4. 

\begin{lemma}\label{lemma_N_luecking}
	Let $C>1$ and an $(r,\kappa)$-lattice $Z=\{z_k\}$. There is $N$ such that for all $n\in \N$ and $\xi\in\T$ there are $N$ hyperbolic discs $D(z_k,r)$ that cover the set
	$$S_{C}(\xi)\cap \left\{z\in\D\ :\ \frac{1}{2^{n+1}}\leq 1-|z|\leq \frac{1}{2^n}\right\}.$$
\end{lemma}
%

In order to facilitate the reading, we have included the following lemma due to Wu.

\begin{lemma}\cite[Lemma 2.3, p. 992]{Wu_2006}\label{Wu} Let $M>1$, $r\geq 0$ and $\xi\in\T$. If $M_\ast=(M+1)e^{2r}-1$, then
$$D(z,r)\subset S_{M_\ast}(\xi)$$
for all $z\in S_{M}(\xi)$.
\end{lemma}

 A standard way to get an $(r,\kappa)$-lattice is through the Whitney Decomposition of the unit disc. This is  the $(r,\kappa)$-lattice described in the following example and the one that Luecking makes use of in \cite{Luecking}.   \newline

\begin{example}\label{LueckingCenters}
Let $Z=\{z_{i,j}\}$ be the sequence formed by the centers of the regions 
$$R_{i,j}=\left\{z\in\D\ :\ 1-\frac{1}{2^{i}}\leq |z|<1-\frac{1}{2^{i+1}}\  ,\arg(z)\in\left[\frac{2\pi j}{2^{i}},\frac{2\pi (j+1)}{2^{i}}\right)\right\}$$
for $i\in\N\cup\{0\}$ and $j=0,1,\dots, 2^{i}-1$.  We will refer to them as Luecking regions.
There exist $r>\kappa>0$ such that $Z=\{z_{i,j}\}$ is an $(r,\kappa)$-lattice.

\end{example}

Our approach  in order to get our results is mainly based on the discrete version of the tent spaces. These are the tent spaces of sequences.
\begin{definition} Let $Z=\{z_{n}\}$ be a $(r,\kappa)$-lattice and $0<p,q<+\infty$.
	We say that $\{\lambda_n\}\in T_{p}^{q}(Z)$ if
	\begin{align}\label{TentsequenceA}
	\|\{\lambda_n\}\|_{T_{p}^{q}(Z)}:=\left(\int_{\T} \left(\sum_{z_n\in \Gamma(\xi)} |\lambda_n|^{p}\right)^{p/q}\ |d\xi|\right)^{1/q}<+\infty,
	\end{align}
	the sequence $\{\lambda_n\}_{n}\in T_{\infty}^{q}(Z)$ if
	\begin{align}\label{TentsequenceB}
	\|\{\lambda_n\}\|_{T_{\infty}^{q}(Z)}:=\left(\int_{\T} \left(\sup_{z_n\in \Gamma(\xi)} |\lambda_n|^{p}\right)^{p/q}\ |d\xi|\right)^{1/q}<+\infty,
	\end{align}
	and $\{\lambda_n\}\in T_{p}^{\infty}(Z)$ if
	\begin{align}\label{TentsequenceC}
	\|\{\lambda_n\}\|_{T_{p}^{\infty}(Z)}=\sup_{\xi\in \T} \left(\sup_{u\in \Gamma(\xi)} \frac{1}{(1-|u|^2)}\sum_{z_{n}\in S(u)} |\lambda_n|^{p}(1-|z_n|^2)\right)^{1/p}<+\infty.
	\end{align}
	where 
	$$S(re^{i\theta})=\left\{\rho e^{it}\ : 1-\rho\leq 1-r,\ |t-\theta|\leq \frac{1-r}{2}\right\}$$ 
	 for $re^{i\theta}\in \D\setminus\{0\}$ and $S(0)=\D$.
\end{definition}
\begin{remark}\label{independetregionsequence}
	Lemma~\ref{estimate 1}  justifies the independence of the definition  \eqref{TentsequenceA} for any $\Gamma_M (\xi)$.  
\end{remark}

The following lemma is a well-know result in the theory of tent spaces of sequences and it will be useful in order to obtain equivalent expressions when we consider different non-tangential regions.
\begin{lemma}\label{lemma_sup} Let $M_1>M>1/2$. There is a constant $C>0$ such that
	\begin{align*}
	\int_{\T} \left(\sup_{z_k\in S_{M_1}(\xi)} |\lambda_{k}|^{q}\right)\ |d\xi|\leq C \int_{\T} \left(\sup_{z_k\in S_{M}(\xi)} |\lambda_{k}|^{q}\right)\ |d\xi|
	\end{align*}
	for all $0<q<+\infty$, and for all sequences $\{z_k\}\subset \D$ and $\{\lambda_{k}\}\subset \C$.
\end{lemma}

The next version of the sub-mean value property will be useful several times. As usual we write $f^{(0)}=f$.

\begin{lemma}\cite[p. 338 ]{Luecking}(See also \cite[Corollary 1, p. 68]{duren_schuster_2004} for the case $n=0$)\label{MVP} If $0<p<+\infty$, $r>0$, $n\in\N\cap \{0\}$, and $f$ is analytic in $\D$, then 
	$$|f^{(n)}(\alpha)|^{p}\lesssim \frac{1}{(1-|\alpha|)^{2+np}}\int_{D(\alpha,r)} |f(w)|^{p}\ dm(w),$$
for each point $\alpha\in \D$.	
\end{lemma}

The relationship between the discrete version of tent spaces and the continuous
version is given in the next proposition.

\begin{proposition}\label{sequence and tent}
Let $0<p,q<+\infty$ and $Z=\{z_k\}$ be an $(r,\kappa)$-lattice. Given $f\in AT_{p}^{q}(\alpha)$ and $\lambda_{k}=\lambda_{k}(f)=\sup_{w\in\overline{D(z_k,r)}}|f(w)|(1-|z_k|)^{\alpha/p}$. Then
$$\|f\|_{T_{p}^{q}(\alpha)}\asymp \| \lambda_{k}\|_{T_{p}^{q}(Z)}.$$
\end{proposition}

\begin{proof}
Take $\tilde{z}_k\in \overline{D(z_k,r)}$ such that $|f(\tilde{z}_k)|=\sup_{w\in\overline{D(z_k,r)}} |f(w)|$.
First, let us check that $\| \{|f(\tilde{z}_k)|(1-|z_k|)^{\alpha/p}\}\|_{T_{p}^{q}(Z)} \lesssim \|f\|_{T_{p}^{q}(\alpha)}$.
Applying the mean value property over the hyperbolic disc $D(\tilde{z}_k,s)$ with $s<3r$, we obtain
\begin{align*}
&\| \{|f(\tilde{z}_{k})|(1-|z_k|)^{\alpha/p}\}\|_{T_{p}^{q}(Z)} =\left(\int_{\T}\left(\sum_{z_k\in \Gamma_M(\xi)}|f(\tilde{z}_{k})|^{p}(1-|z_k|)^{\alpha}\right)^{q/p} |d\xi|\right)^{1/q}\\
&\quad\lesssim \left(\int_{\T}\left(\sum_{z_k\in \Gamma_M(\xi)}\left(\int_{D(\tilde{z}_k,s)} |f(z)|^{p} \frac{\ dm(z)}{(1-|\tilde{z}_k|)^{2}}\right)(1-|z_k|)^{\alpha}\right)^{q/p} |d\xi|\right)^{1/q}.
\end{align*}
Since $D(\tilde{z}_k,s)\subset D(z_k,4r)$, applying Remark~\ref{remarkestimatedistbound} we have
\begin{align*}
&\| \{|f(\tilde{z}_{k})|(1-|z_k|)^{\alpha/p}\}\|_{T_{p}^{q}(Z)}\\
&\quad\lesssim \left(\int_{\T}\left(\sum_{z_k\in \Gamma_M(\xi)}\left(\int_{D({z}_k,4r)} |f(z)|^{p} \frac{\ dm(z)}{(1-|{z}_k|)^{2-\alpha}}\right)\right)^{q/p} |d\xi|\right)^{1/q}.
\end{align*}
 By Lemma~\ref{Wu}, we can take $M_{+}>M>1/2$ such that
\begin{align*}
\bigcup_{D(z_{k},4r)\cap \Gamma_M (\xi)\neq \emptyset } D(z_k,4r)\subset \Gamma_{M_+}(\xi),
\end{align*}
and applying Remark~\ref{remarkestimatedistbound}, it follows that
\begin{align*}
&\| \{|f(\tilde{z}_{k})|(1-|z_k|)^{\alpha/p}\}\|_{T_{p}^{q}(Z)}\\
&\quad\lesssim \left(\int_{\T}\left(\sum_{z_k\in \Gamma_M(\xi)}\left(\int_{D({z}_k,4r)} |f(z)|^{p} \frac{\ dm(z)}{(1-|{z}|)^{2-\alpha}}\right)\right)^{q/p} |d\xi|\right)^{1/q}\\
&\quad \leq  \left(\int_{\T}\left(\int_{\Gamma_{M_+}(\xi)}\left(\sum_{z_k\in \Gamma(\xi)}\chi_{D({z}_k,4r)}(z)\right) |f(z)|^{p} \frac{\ dm(z)}{(1-|{z}|)^{2-\alpha}}\right)^{q/p} |d\xi|\right)^{1/q}.
\end{align*}
Due to the fact that $Z=\{z_k\}$ is an $(r,\kappa)$-lattice, by Proposition~\ref{Ncoverdiscs} there is $N$ such that 
$$\sum_{k} \chi_{D({z}_k,4r)}(z)\leq N$$
for all $z\in \D$. Therefore, we obtain
\begin{align*}
&\| \{|f(\tilde{z}_{k})|(1-|z_k|)^{\alpha/p}\}\|_{T_{p}^{q}(Z)}\\
&\quad \lesssim N^{1/p} \left(\int_{\T}\left(\int_{\Gamma_{+}(\xi)} |f(z)|^{p} \frac{\ dm(z)}{(1-|{z}|)^{2-\alpha}}\right)^{q/p} |d\xi|\right)^{1/q} \lesssim N^{1/p}\|f\|_{T_{p}^{q}(\alpha)}.
\end{align*}

Now, we proceed with the converse inequality. Using the pointwise estimate $|f(z)|\leq \sum_{k} |f(z)| \chi_{D({z}_k,r)}(z)$ and the fact (given by Lemma~\ref{Wu}) that we can take $M_{+}>M>1/2$ such that
$$
\bigcup_{D(z_{k},r)\cap \Gamma (\xi)\neq \emptyset } D(z_k,r)\subset \Gamma_{M_+}(\xi),
$$
it follows that
\begin{align*}
\|f\|_{T_{p}^{q}(\alpha)}&=\left(\int_{\T}\left(\int_{\Gamma(\xi)} |f(z)|^{p} \frac{\ dm(z)}{(1-|{z}|)^{2-\alpha}}\right)^{q/p} |d\xi|\right)^{1/q}\\
&\leq\left(\int_{\T}\left(\sum_{k}\int_{\Gamma(\xi)\cap D({z}_k,r)} |f(z)|^{p}  \frac{\ dm(z)}{(1-|{z}|)^{2-\alpha}}\right)^{q/p} |d\xi|\right)^{1/q}\\
& \leq \left(\int_{\T}\left(\sum_{z_{k}\in \Gamma_{+}(\xi)}\int_{D({z}_k,r)} |f(z)|^{p}  \frac{\ dm(z)}{(1-|{z}|)^{2-\alpha}}\right)^{q/p} |d\xi|\right)^{1/q}.
\end{align*}
Taking $|f(\tilde{z}_k)|:=\sup_{w\in\overline{D(z_k,r)}} |f(w)|$ and applying Remark~\ref{remarkestimatedistbound}, we have
\begin{align*}
&\|f\|_{T_{p}^{q}(\alpha)}\lesssim \left(\int_{\T}\left(\sum_{z_{k}\in \Gamma_{+}(\xi)}|f(\tilde{z}_k)|^{p} (1-|z_k|)^{\alpha} \right)^{q/p} |d\xi|\right)^{1/q}\asymp\| \{|f(\tilde{z}_{k})|(1-|z_k|)^{\alpha/p}\}\|_{T_{p}^{q}(Z)}.
\end{align*}
Therefore, we are done.
\end{proof}

In \cite{Cohn_Verbitsky_200} Cohn and Verbitsky proved a result about the factorization of tent spaces of functions over the upper half-space, but we will use the following concerning the factorization of sequence tent spaces (see \cite[Proposition 6, p. 19]{MiihkinenPauPeralaWang2020}) throughout this article.

\begin{proposition}\label{FactorizationTentSequence}
	Let $0<p,q<+\infty$ and $Z=\{a_k\}$ be an $(r,\kappa)$-lattice. If $p\leq p_1,p_2\leq +\infty$, $q\leq q_1,q_2\leq +\infty$ and satisfy $\frac{1}{p}=\frac{1}{p_1}+\frac{1}{p_2}$, and $\frac{1}{q}=\frac{1}{q_1}+\frac{1}{q_2}$. Then $$T_p^q(Z)=T_{p_1}^{q_1}(Z)\cdot T_{p_2}^{q_2}(Z).$$
\end{proposition}
\begin{remark}
	In \cite[Proposition 6, p. 19]{MiihkinenPauPeralaWang2020} this result is stated for the cases $p<p_1,p_2<+\infty$ and $q<q_1,q_2<+\infty$. However, using the same argument we can also extend it to extreme cases.
\end{remark}

The following propositions give us the duality for the tent spaces of sequences. They will be a cornerstone in the proof of the main result of this paper.
\begin{proposition}\cite[Lemma 6, p. 68]{Arsenovic}\label{Arsenovic1}
	Let $Z=\{z_n\}$ be a $(r,\kappa)$-lattice and $1\leq p<+\infty$, $1<q<+\infty$. Then $(T_{p}^{q}(Z))^{\ast}\cong T_{p'}^{q'}(Z)$, where $\frac{1}{p}+\frac{1}{p'}=1$ and $\frac{1}{q}+\frac{1}{q'}=1$. The isomorphism between $(T_{p}^{q}(Z))^{\ast}$ and $ T_{p'}^{q'}(Z)$ is given by the operator 
	$$\{b_k\}\mapsto \langle \cdot,\{b_k\}\rangle$$
	where $\langle \cdot,\{b_k\}\rangle$ is defined by
	\begin{align*}
	\langle \{a_k\},\{b_k\}\rangle=\sum_{k} a_{k}\ b_{k} (1-|z_{k}|),\quad \{a_k\}\in T_{p}^{q}(Z).
	\end{align*}
	In fact, $\|\{b_k\}\|_{T_{p'}^{q'}(Z)}\asymp \sup\left\{\left|\sum_{k} a_kb_k(1-|z_k|)\right|\ :\ \|\{a_k\}\|_{T_{p}^{q}(Z)}=1\right\}$. 
\end{proposition}

\begin{proposition}\cite[Poposition 2, p. 72]{Arsenovic}\label{Arsenovic2}
	Let $Z=\{z_n\}$ be a $(r,\kappa)$-lattice and  $0<p<1<q<+\infty$. Then $(T_{p}^{q}(Z))^{\ast}\cong T_{\infty}^{q'}(Z)$, where $\frac{1}{q}+\frac{1}{q'}=1$. The isomorphism between $(T_{p}^{q}(Z))^{\ast}$ and $T_{\infty}^{q'}(Z)$
	is given by the operator as in Proposition~\ref{Arsenovic1}.
	In fact, $\|\{b_k\}\|_{T_{\infty}^{q'}(Z)}\asymp \sup\left\{\left|\sum_{k} a_kb_k(1-|z_k|)\right|\ :\ \|\{a_k\}\|_{T_{p}^{q}(Z)}=1\right\}$.
\end{proposition}

\begin{proposition}{\cite[Lemma 3.4, p. 184]{Jevtic_1996}}\label{Jevtic}
	Let $Z=\{z_n\}$ be a $(r,\kappa)$-lattice and $1<p<+\infty$. Then $(T_{p}^{1}(Z))^{\ast}\cong T_{p'}^{\infty}(Z)$, where $\frac{1}{p}+\frac{1}{p'}=1$. The isomorphism between $(T_{p}^{1}(Z))^{\ast}$ and $T_{p'}^{\infty}(Z)$
	is given by the operator as in Proposition~\ref{Arsenovic1}.
	
		In fact, $\|\{b_k\}\|_{T_{p'}^{\infty}(Z)}\asymp \sup\left\{\left|\sum_{k} a_kb_k(1-|z_k|)\right|\ :\ \|\{a_k\}\|_{T_{p}^{1}(Z)}=1\right\}$. 
\end{proposition}

Following the same argument of Luecking in \cite[Proposition 2, p. 352]{Luecking} and combining Lemma~\ref{lemma_N_luecking}, we have the same result of tent spaces of sequences in $\D$ now for any $(r,\kappa)$-lattice, but avoiding the cases $1<p<+\infty$, $q=1$. So that, we decide to omit the proof.

\begin{proposition}\label{Luecking1}
	Let $Z=\{z_k\}$ be an $(r,\kappa)$-lattice. If either $0<p<+\infty$, $0<q<1$ or $0<p\leq 1$, $q=1$, then  
	$$
	\sup\left\{\left|\sum_{k} a_kb_k(1-|z_k|)\right|\ :\ \|\{b_k\}\|_{T_p^q(Z)}=1\right\}\asymp\sup_{k} |a_k|(1-|z_k|)^{1-1/q}$$
	for any sequence $\{a_k\}$.
\end{proposition}

\begin{remark}
We point out that in \cite[Proposition 2, p. 352]{Luecking} above result for the $(r,\kappa)$-lattice given in Example~\ref{LueckingCenters} is stated for $0<p<+\infty$, but if $p>1$ and $q=1$, its proof does not work.
\end{remark}

The following result, due to Perälä, will be an important tool in the characterization of the Carleson measures in the next section. 

\begin{proposition}{\cite[Lemma 14, p. 24]{Perala_2018}}\label{Soperator}
	Let $Z=\{z_n\}$ be an $(r,\kappa)$-lattice, $0< p,q<+\infty$, $\alpha>0$, and  $M>\max\left\{1,p/q,1/q,1/p\right\}+\alpha/p$. Then the operator $S: T_{p}^{q}(Z)\rightarrow AT_{p}^{q}(\alpha)$, where
	\begin{align*}
	S(\{\lambda_k\})(z):=\sum_{k=0}^{\infty} \lambda_k \frac{(1-|z_k|)^{M-\frac{\alpha}{p}}}{(1-\overline{z_k}z)^{M}},
	\end{align*}
	is bounded.
\end{proposition}

\section{The tent space nature of the $RM(p,q)$}
In  this section  we present equivalent representations of the $RM(p,q)$ spaces in terms of the derivatives of their elements and as tent spaces. 
It seems that the $RM(p,q)$ and their equivalent representations  form part of the theory of Triebel spaces. 
As a
reference to that we propose  \cite{Cohn_Verbitsky_200}.
However, neither  is  clear nor is easy  to justify the equivalent formulas  through  that approach. 

Our starting point is the 
 Littlewood-Paley theory  for the Hardy spaces $H^q\, (q>0)$,
according to which
\begin{equation*}
\int_0^{2\pi} \left( \int_0^1 |f'(re^{i\theta})|^2 (1-r)\,dr \right)^{\frac q2} \,d\theta \asymp \|f\|^q_{H^q} 
\asymp \int_{\mathbb T}\,\left(\int_{\Gamma(\xi)}\,|f'(z)|^2\, dA(z)\right)^{\frac p2}\, |d\xi|\,,
\end{equation*}
and \cite{Aguilar-Contreras-Piazza_2} where the authors prove that, for $1< p,q <\infty\,$ or $(1,q)$ with $1\leq q<+\infty$,
\begin{equation}\label{CL}
\rho_{p,q}(f)^{q}\asymp \int_0^{2\pi}\,\left(\int_0^1\, |f'(re^{i\theta})|^p (1-r)^p\, dr\right)^{\frac qp}\,d\theta\,.
\end{equation}
As we shall see the equivalence (\ref{CL}) extends to the case $RM(p,1),\, p>1\,.$

 First we establish how the $RM(p,q)$ spaces can be represented as tent spaces.
Our approach is based on the ideas of \cite{Pavlovic} appropriately applied to our more general case. 
\begin{proposition}\label{non_tangential_char_without_der}
	Let $1\leq p,q<+\infty$. Then for $f\in\mathcal{H}(\D)$ we have that 
	\begin{enumerate}
		\item $\rho_{p,q}(f)\asymp \|f\|_{T_{p}^{q}(1)},$
		\item $\rho_{p,q}(f'(\cdot)(1-|\cdot|))\asymp \| f'(\cdot) (1-|\cdot|)\|_{T_p^q (1)}.$
		\newline
	\end{enumerate}

	That is,\,\, $RM(p,q)=AT_{p}^{q}(1)$.
\end{proposition}

\begin{proof}
	$(1)$ First we show that 
	\begin{equation}\label{charineqA}
\rho_{p,q}(f)\lesssim \left(\int_{0}^{2\pi} \left(\int_{\Gamma(e^{i\theta})} |f(w)|^{p}\ \frac{dA(w))}{1-|w|}\right)^{q/p} \frac{d\theta}{2\pi}\right)^{1/q}.
	\end{equation}
	Fixed $\theta\in [0,2\pi]$, we have that
	\begin{align*}
	\int_{0}^{1} |f(re^{i\theta})|^{p} \ dr &=\sum_{k=0}^{\infty} \int_{1-2^{-n}}^{1-2^{-(n+1)}}  |f(re^{i\theta})|^{p} \ dr\\
	&\leq \sum_{k=0}^{\infty} \sup_{1-2^{-n}<r<1-2^{-(n+1)}} |f(re^{i\theta})|^{p} (2^{-n}-2^{-(n+1)})\\
	&= \sum_{k=0}^{\infty} 2^{-(n+1)} \sup_{2^{-(n+1)}<1-r<2^{-n}} |f(re^{i\theta})|^{p}\\
	&\lesssim  \sum_{k=0}^{\infty}  \frac{2^{-(n+1)}}{2^{-2n}} \int_{E_n(\theta)} |f(w)|^{p}\ dA(w),
	\end{align*}
	where $E_n(\theta)=\{w\in\D: |w-(1-2^{-n})e^{i \theta}|<\frac{3}{2^{n+2}} \}$. Moreover, it is easy to see that $E_j(\theta)\cap E_{k}(\theta)=\emptyset$, if $|j-k|\geq 3$, and that there is a constant $ M>\frac 12$ such that $E_{n}(\theta)\subset \Gamma_M (e^{i\theta})$ for all $n\in\N$. Therefore, it follows that 
	\begin{align*}
	\int_{0}^{1} |f(re^{i\theta})|^{p} \ dr & \lesssim \sum_{k=0}^{\infty}  2^{n} \int_{E_n(\theta)} |f(w)|^{p} \frac{(1-|w|)}{(1-|w|)}\ dA(w)\\
	&\lesssim \sum_{k=0}^{\infty}   \int_{E_n(\theta)} |f(w)|^{p} \frac{dA(w)}{1-|w|}\lesssim \int_{\Gamma_M (e^{i\theta})} |f(w)|^{p} \frac{dA(w)}{1-|w|}
	\end{align*}
	From this, we clearly have \eqref{charineqA}.
	
	Now,  for the comparison from below, we continue with the proof of 
	\begin{align*}
	\left(\int_{\mathbb T} \left(\int_{\Lambda (\xi)} |f(w)|^{p}\ \frac{dA(w)}{1-|w|}\right)^{q/p} |d\xi| \right)^{1/q}\lesssim  \rho_{p,q}(f).
	\end{align*}
	
	Notice that for $\xi\in\T$
	\begin{align*}
	\int_{\Lambda (\xi)} & |f(w)|^{p}\ \frac{dA(w)}{1-|w|}\\
	&=\int_{0}^{1} \int_{|\theta|<1-r} |f(re^{i\theta}\xi)|^{p} (1-r)^{-1} d\theta \ dr \\
	&\leq \sum_{n=0}^{\infty} \int_{1-2^{-n}}^{1-2^{-(n+1)}} \int_{|\theta|<2^{-n}} |f(re^{i\theta}\xi)|^{p}\ (1-r)^{-1}\ d\theta  dr\\
	&\leq \sum_{n=0}^{\infty}  \int_{|\theta|<2^{-n}} \int_{1-2^{-n}}^{1-2^{-(n+1)}} |f(re^{i\theta}\xi)|^{p} (1-r)^{-1}\ dr \ d\theta\\
	& \leq\sum_{n=0}^{\infty} 2^{n+1}  \int_{|\theta|<2^{-n}} h_{n}((1-2^{-n-2})\xi e^{i\theta})\ d\theta\,, \\
	\end{align*}
	where
	$$h_n(z)=\int_{1-2^{-n}}^{1-2^{-(n+1)}} \left|f\left(\frac{rz}{1-2^{-n-2}}\right)\right|^p \ dr.$$
	For each $n$, the function $h_n$ is log-subharmonic (see \cite[p. 36]{Ronkin_1974}).
	 Bearing in mind  that 
	 $$|(1-2^{-n-2})\xi e^{i\theta}-\xi|\leq M  2^{-n-2}\,, \quad \text{for}\, \, |\theta|<2^{-n}\,,$$
	 where $M$ is an absolute constant, it follows that $(1-2^{-(n+2)})\xi e^{i\theta}\in \Gamma_{M}(\xi)$ and
	\begin{align*}
	\int_{0}^{1} \int_{|\theta|<1-r} |f(re^{i\theta}\xi)|^{p} (1-r)^{-1}\ d\theta \ dr \leq 4\sum_{n=0}^{\infty}  M_{\ast}h_{n}(\xi)
	\end{align*}
	where $M_{\ast}$ denotes the non-tangential maximal operator, that is 
	 $$M_{\ast}h_n(\xi)=\sup_{z\in \Gamma_{M}(\xi)} h_n(z)\,.$$
	
	Applying \cite[Theorem 1.8, p. 173]{Pavlovic}, we find a positive constant $C(p,q)$ such that 
	\begin{align*}
	&\left(\int_{\mathbb T} \left(\int_{\Lambda(\xi)} |f(w)|^{p}\ \frac{dA(w)}{1-|w|}\right)^{q/p} |d\xi| \right)^{1/q}
	 \leq \,\,\left(\int_{\T}\left(4\sum_{n=0}^{\infty}  M_{\ast}h_{n}(\xi)\right)^{q/p}\ |d\xi|\right)^{1/q}\\
	&\leq \,\, C(p,q)\sup_{0\leq s<1} \left(\int_{\T}\left(\sum_{n=0}^{\infty}  h_{n}(s\xi)\right)^{q/p}\ |d\xi|\right)^{1/q}\\
	&= \,\,C(p,q) \sup_{0\leq s<1} \left(\int_{\T}\left(\sum_{n=0}^{\infty}  \int_{1-2^{-n}}^{1-2^{-(n+1)}} \left|f\left(\frac{rs\xi}{1-2^{-n-2}}\right)\right|^p\ dr\right)^{q/p}\ |d\xi|\right)^{1/q}.
	\end{align*}
	Doing the change of variable $u=\frac{r}{1-2^{-(n+2)}}$, it follows that
	\begin{align*}
	&\left(\int_{\mathbb T} \left(\int_{\Lambda (\xi)} |f(w)|^{p}\ \frac{dA(w)}{1-|w|}\right)^{q/p} |d\xi| \right)^{1/q}\\
	&\leq \,\, C(p,q) \sup_{0\leq s<1}\left(\int_{\T}\left(\sum_{n=0}^{\infty} 2^{-n} \int_{1-2^{-n}}^{1-2^{-(n+2)}} \left|f\left(rs\xi\right)\right|^p \ dr\right)^{q/p}\ |d\xi|\right)^{1/q}\\
	&\leq \, 2\,\, C(p,q) \sup_{0\leq s<1}\rho_{p,q}(f_s).
	\end{align*}
	By \cite[Proposition 2.12]{Aguilar-Contreras-Piazza}, we conclude that
	\begin{align*}
	\left(\int_{\mathbb T} \left(\int_{\Lambda (\xi)} |f(w)|^{p}\ \frac{dA(w)}{1-|w|}\right)^{q/p} |d\xi| \right)^{1/q}\lesssim  \rho_{p,q}(f).
	\end{align*}
	
	$(2)$ \,Following the same argument, we obtain the equivalent result for the derivative.
	
\end{proof}

 The next theorem is due to Perälä \cite{Perala_2018}. It is the equivalent description of the   $AT_p^q (1)$ in terms of the derivative. 

\begin{theorem}\cite[Theorem 2, p. 9]{Perala_2018}\label{characterization 2}
Let $0<p,q,\alpha<+\infty$. Then, we have that
\begin{align*}
\|f'(\cdot)(1-|\cdot|)\|_{T_p^q(\alpha)}\asymp  \|f\|_{T_p^q(\alpha)},\ f\in\mathcal{H}(\D).
\end{align*}
\end{theorem}

Combining the above we get  the following corollary.

\begin{corollary}
	Let $1\leq p,q<+\infty$. Then, we have that
\begin{align*}
\rho_{p,q}(f)\asymp  \rho_{p,q}(f'(\cdot)(1-|\cdot|)),\ f\in\mathcal{H}(\D).
\end{align*}
\end{corollary}

\section{Carleson measures}
This section is devoted to the study of Carleson measure type problems for the tent spaces of analytic functions  and  consecuentely for the $RM(p,q)$ spaces. 
First, we present our main result. It answers the problem posed by Luecking in \cite{Luecking}. See section $7$.

\begin{theorem}\label{theoremCarlesonTent}
	Let $0< p,q,s,t, \alpha<+\infty$,  $M>1/2$, $Z=\{z_k\}$ an $(r,\kappa)$-lattice, and let $\mu$ be a positive Borel measure on $\D$. Then the following are equivalent:
	\begin{enumerate}
		\item There is a constant $C>0$ such that 
		\begin{align*}
		\left(\int_{\T} \left(\int_{\Gamma (\xi)} |f(w)|^{t}\ d\mu(w)\right)^{s/t}\ |d\xi|\right)^{1/s}\leq C \|f\|_{T_{p}^{q}(\alpha)}, \quad f\in AT_{p}^{q}(\alpha).
		\end{align*}
		\item The measure $\mu$ satisfies the following:
		\begin{enumerate}
			\item If $0< s<q<+\infty$, $0< t<p<+\infty$, then
			\begin{align*}
			\int_{\T} \left(\sum_{z_k\in \Gamma (\xi)} \left(\frac{\mu^{1/t}(D(z_k,r))}{(1-|z_k|)^{\alpha/p}}\right)^{\frac{pt}{p-t}} \right)^{\frac{(p-t)qs}{(q-s)pt}}\ |d\xi|<+\infty.
			\end{align*}
			\item If $0< s< q<+\infty$, $0< p\leq t<+\infty$, then
			\begin{align*}
			\int_{\T} \left(\sup_{z_k\in \Gamma (\xi)} \frac{\mu^{1/t}(D(z_k,r))}{(1-|z_k|)^{\alpha/p}} \right)^{\frac{qs}{q-s}}\ |d\xi|<+\infty.
			\end{align*}
			\item If $0< q< s<+\infty$, $0< p,t<\infty$ or $0< q=s<+\infty$, $0< p\leq t<+\infty$, then
			\begin{align*}
			\sup_{k} \frac{\mu^{1/t}(D(z_k,r))}{(1-|z_k|)^{\frac{\alpha}{p}+\frac{1}{q}-\frac{1}{s}}}<+\infty.
			\end{align*}
			\item If $0< q=s<+\infty$, $0< t< p<+\infty$, then
			\begin{align*}
			\sup_{\xi\in\T} \left(\sup_{\xi\in I}  \sum_{z_{k}\in T(I)} \left(\frac{\mu^{1/t}(D(z_k,r))}{(1-|z_k|)^{\alpha/p}}\right)^{\frac{pt}{p-t}} (1-|z_k|)\  \right)^{\frac{p-t}{pt}}<+\infty,
			\end{align*}
			where $I$ runs the intervals in $\T$, $S(I)=\left\{z\in\D\ :\ 1-|I|\leq |z|<1\ \text{and}\ \frac{z}{|z|}\in I\right\}$ and $|I|$ is its arc length.
		\end{enumerate}
	\end{enumerate}
\end{theorem}

\begin{proof}
	First, we will prove that \emph{(2)} implies \emph{(1)}.  Let $b>\max\{\frac{1}{t},\frac{1}{s}\}$.
 By Lemma~\ref{Wu}, one can take a constant $M_{+}>M>\frac 12 $ such that, for all $\xi\in\T$,
	$$
	\bigcup_{D(z_{k},r)\cap \Gamma_{M} (\xi)\neq \emptyset } D(z_k,r) \subset \Gamma_{M_{+}}(\xi)
	$$
	where  
	$D(z_k,r)$ are hyperbolic disks that correspond to the $(r,\kappa)$-lattice.\\
		\\
	Fix an $f\in AT_{p}^{q}(\alpha)$. Using the pointwise estimate $|f(z)|\leq \sum_{k} |f(z)|\chi_{D(z_k,r)}(z)$, it follows that
	\begin{align*}
&\left( \int_{\T} \left(\int_{\Gamma_{M}(\xi)} |f(w)|^{t}\ d\mu(w)\right)^{s/t}\ |d\xi| \right)^{\frac {1}{s}}
	\leq \left( \int_{\T} \left(\sum_{k} \int_{D(z_k,r)\cap \Gamma_{M} (\xi)}|f(w)|^{t} \ d\mu(w)\right)^{s/t}\ |d\xi|\right)^{\frac {1}{s}}\\
	&\quad\leq \left( \int_{\T} \left(\sum_{z_k\in \Gamma_{M_+}(\xi)} \, \sup_{w\in \overline{D(z_k,r)} }|f(w)|^{t} \mu (D(z_k,r))\ d\mu(w)\right)^{s/t}\ |d\xi|\right)^{\frac {1}{s}}.
	\end{align*}
	For the simplicity of the presentation we set $|f_k|:=\sup_{w\in \overline{D(z_k,r)}} |f(w)|$.  Bearing in mind Lemma~\ref{estimate 1} and $b>\max\{\frac{1}{t},\frac{1}{s}\}$, we have
	\begin{align}\label{ineq1}
	&\left(\int_{\T} \left(\int_{\Gamma_{M} (\xi)} |f(w)|^{t}\ d\mu(w)\right)^{s/t}\ |d\xi| \right)^{\frac {1}{s}}\\\nonumber
	 &\quad\leq \left(\int_{\T} \left(\sum_{z_k\in\Gamma_{M_+}(\xi)} \,|f_k|^{t} \mu (D(z_k,r))\ d\mu(w)\right)^{s/t}\ |d\xi| \right)^{\frac {1}{s}}\\\nonumber
	& \quad = \left(\int_{\T} \left(\sum_{z_k\in \Gamma_{M_+}(\xi)} \left(|f_k|^{1/b}\ \mu^{1/bt}(D(z_k,r)) \right)^{bt}\right)^{bs/bt}\ |d\xi|\right)^{\frac{b}{bs}}\\\nonumber
	&\quad \asymp \|\{|f_k|^{1/b}\, \mu^{1/bt} (D(z_k,r))\}\|_{T^{bs}_{bt}(Z)}^ b\,.\\\nonumber
	\end{align}
	 \\
	Using the duality relation for the $T^{bs}_{bt}(Z)$ as stated in Proposition~\ref{Arsenovic1}, we get that 
	\begin{align}\label{ineq2}
	&\|\{|f_k|^{1/b}\, \mu^{1/bt} (D(z_k,r))\}\|_{T^{bs}_{bt}(Z)}\\
	 &\quad= \sup_{\substack{\lambda_{k}\geq 0\\\nonumber \|\{\lambda_{k}\}\|_{T_{(bt)'}^{(bs)'}(Z)}\leq 1}}\sum_{k} \lambda_{k} |f_k|^{1/b}\,\mu^{1/bt}(D(z_k,r))\, (1-|z_k|)\\\nonumber
	& \quad=\sup_{\substack{\lambda_{k}\geq 0\\ \|\{\lambda_{k}\}\|_{T_{(bt)'}^{(bs)'}(Z)}\leq 1}}\sum_{k} \lambda_{k} |f_k|^{1/b} (1-|z_k|)^{\alpha /bp} \frac{\mu^{1/bt}(D(z_k,r))}{(1-|z_k|)^{\alpha /bp}} (1-|z_k|).
	\end{align}
	
	By Remark~\ref{sequence and tent}, we have
	\begin{align*}
	\|f \|_{ T^{q}_{p} (\alpha)} \asymp \|\{|f_k| (1-|z_k|)^{\alpha/p}\}\|_{T_{p}^{q}(Z)}= \|\{|f_k|^{1/b} (1-|z_k|)^{\alpha/bp}\}\|^b_{T_{bp}^{bq}(Z)}\,.
	\end{align*}
	Therefore, for any  sequence of positive numbers $\{ \lambda_k \} \in T_{(bt)'}^{(bs)'}(Z)$,
	 Proposition~\ref{FactorizationTentSequence} implies that
	  $$
	  \{ \lambda_{k} |f_k|^{1/b} (1-|z_k|)^{\alpha/bp} \}\in T_{(bt)'}^{(bs)'}(Z)\cdot T_{bp}^{bq}(Z)= T_{\frac{pbt}{t+p(bt-1)}}^{\frac{qbs}{s+q(bs-1)}}(Z),
	  $$ 
	  
	Taking into account the conditions \emph{(2)(a)-(d)} together with Propositions~\ref{Arsenovic1}--\ref{Luecking1} respectively, we obtain that there is a constant $C(\mu)>0$ such that
	\begin{align}\label{ineq3}
	&\left(\sum_{k} \lambda_{k} |f_k|^{1/b} (1-|z_k|)^{\alpha /bp} \frac{\mu^{1/bt}(D(z_k,r))}{(1-|z_k|)^{\alpha /bp}} (1-|z_k|)\right)^{b}\\\nonumber
	&\quad\leq C(\mu)^{b} \|  \{\lambda_{k} |f_k|^{1/b} (1-|z_k|)^{\alpha/bp}\}\|_{T_{\frac{pbt}{t+p(bt-1)}}^{\frac{qbs}{s+q(bs-1)}}(Z)}^{b}\\\nonumber
	&\quad\lesssim C(\mu)^{b} \,\| \{\lambda_k\} \|_{T_{(bt)'}^{(bs)'}(Z)}^{b}\, \|\{|f_k|^{1/b} (1-|z_k|)^{\alpha/bp}\}\|^b_{T_{bp}^{bq}(Z)}\\\nonumber
	&\quad = C(\mu)^{b} \,\| \{\lambda_k\} \|_{T_{(bt)'}^{(bs)'}(Z)}^{b}\,  \|\{|f_k| (1-|z_k|)^{\alpha/p}\}\|_{T_{p}^{q}(Z)}\\\nonumber
	& \quad \lesssim C(\mu)^{b} \,\| \{\lambda_k\} \|_{T_{(bt)'}^{(bs)'}(Z)}^{b}\,  \|f \|_{ AT^{q}_{p} (\alpha)}\,.\nonumber
	\end{align}
	Now, we proceed to clarify which is the constant $C(\mu)$ in each of the cases of the statement \emph{(2)}.
	\begin{enumerate}
		\item[(a)]  By Proposition~\ref{Arsenovic1}, it follows that $$C(\mu)\asymp\left\|\left\{\frac{\mu^{1/bt}(D(z_k,r))}{(1-|z_k|)^{\alpha /bp}}\right\}\right\|_{T_{\frac{bpt}{p-t}}^{\frac{bqs}{q-s}}(Z)}=\left\|\left\{\frac{\mu^{1/t}(D(z_k,r))}{(1-|z_k|)^{\alpha /p}}\right\}\right\|_{T_{\frac{pt}{p-t}}^{\frac{qs}{q-s}}(Z)}^{1/b}.$$
		\item[(b)] Using Proposition~\ref{Arsenovic1} for the cases $s<q$, $p=t$ and Proposition~\ref{Arsenovic2} for the case $s<q$, $p<t$ we have that
		$$C(\mu)\asymp\left\|\left\{\frac{\mu^{1/bt}(D(z_k,r))}{(1-|z_k|)^{\alpha /bp}}\right\}\right\|_{T_{\infty}^{\frac{bqs}{q-s}}(Z)}=\left\|\left\{\frac{\mu^{1/t}(D(z_k,r))}{(1-|z_k|)^{\alpha /p}}\right\}\right\|_{T_{\infty}^{\frac{qs}{q-s}}(Z)}^{1/b}.$$
		\item[(c)] By means of Proposition~\ref{Luecking1}, the constant $C(\mu)$ in these cases is determined by
		$$C(\mu)\asymp\sup_{k}\frac{\mu^{1/bt}(D(z_k,r))}{(1-|z_k|)^{\alpha /bp}} (1-|z_k|)^{\frac{q-s}{qbs}}=\left(\sup_{k} \frac{\mu^{1/t}(D(z_k,r))}{(1-|z_k|)^{\frac{\alpha}{p}+\frac{1}{q}-\frac{1}{s}}}\right)^{1/b}.$$
		\item[(d)] In the remaining cases, we apply Proposition~\ref{Jevtic} to obtain
		$$C(\mu)\asymp \left\|\left\{\frac{\mu^{1/bt}(D(z_k,r))}{(1-|z_k|)^{\alpha /bp}}\right\}\right\|_{T_{\frac{bpt}{p-t}}^{\infty}(Z)}=\left\|\left\{\frac{\mu^{1/t}(D(z_k,r))}{(1-|z_k|)^{\alpha /p}}\right\}\right\|_{T_{\frac{pt}{p-t}}^{\infty}(Z)}^{1/b}.$$
 	\end{enumerate}
	
	Therefore, combining \eqref{ineq1}, \eqref{ineq2}, and \eqref{ineq3}, we conclude that
	\begin{align*}
	\left(\int_{\T} \left(\int_{\Gamma (\xi)} |f(w)|^{t}\ d\mu(w)\right)^{s/t}\ |d\xi|\right)^{1/s} \lesssim \|f\|_{T_{p}^{q}(\alpha)}.
	\end{align*}
	
	Now, we continue with the proof of \emph{(1)} implies \emph{(2)}. By Lemma~\ref{Wu}, there is a constant $M_\ast>M>1/2$ such that $D(z,r)\subset \Gamma_{M_\ast}(\xi)$ for all $z\in \Gamma_M(\xi)$.  By hypothesis and Lemma~\ref{estimate 1} there is a constant $C>0$ such that 
	\begin{align}\label{hypo1}
	\left(\int_{\T} \left(\int_{\Gamma_{M_\ast} (\xi)} |f(z)|^{t}\ d\mu(z)\right)^{s/t}\ |d\xi|\right)^{1/s}\leq \,C \,\|f\|_{T_{p}^{q}(\alpha)}
	\end{align}
	for all $f\in AT_{p}^{q}(\alpha)$. Take  $\lambda=\{\lambda_k\}$ a sequence  of $T_{p}^{q}(Z)$ and let $\rho_k: [0,1]\rightarrow \{-1,1\}$ the Radermacher functions.
	Moreover, for $K>\max\left\{1,p/q,1/q,1/p\right\}+\alpha/p$,  we consider the function
	\begin{align*}
	F_x(z):=\sum_{k} \lambda_{k} \rho_k(x) \frac{(1-|z_k|)^{K-\frac{\alpha}{p}}}{(1-\overline{z_k}z)^{K}}\,\,, \quad z\in\mathbb D.
	\end{align*}
	From \eqref{hypo1} and Proposition~\ref{Soperator}, it follows that
	\begin{align*}
	\int_{\T} \left(\int_{\Gamma_{M_\ast} (\xi)} |F_x(z)|^{t}\ d\mu(z)\right)^{s/t}\ |d\xi|\lesssim \|\{\lambda_{k}\}\|_{T_{p}^{q}(Z)}^{s}.
	\end{align*}
	Integrating both sides with respect to $x$, we obtain
	\begin{align*}
	\int_{0}^{1}\int_{\T} \left(\int_{\Gamma_{M_\ast}  (\xi)} |F_x(z)|^{t}\ d\mu(z)\right)^{s/t}\ |d\xi|\ dx\lesssim \|\{\lambda_{k}\}\|_{T_{p}^{q}(Z)}^{s}.
	\end{align*}
	Applying Fubini's theorem and Khinchine-Kahane-Kalton inequality (see \cite[Theorem 2.1, p. 251]{Kalton}), it follows that
	\begin{align*}
	\int_{\T} \left(\int_{0}^{1}\int_{\Gamma_{M_\ast}  (\xi)} |F_x(z)|^{t}\ d\mu(z)\ dx\right)^{s/t}\ |d\xi|\lesssim \|\{\lambda_{k}\}\|_{T_{p}^{q}(Z)}^{s}.
	\end{align*}
	Now, by means of Khinchine's inequality we have
	\begin{align*}
	\int_{\T} \left(\int_{\Gamma_{M_\ast}  (\xi)} \left(\sum_k |\lambda_k|^{2}  \frac{(1-|z_k|)^{2 K-\frac{2\alpha}{p}}}{|1-\overline{z_k}z|^{2 K}}\right)^{t/2}\ d\mu(z)\right)^{s/t}\ |d\xi|\lesssim \|\{\lambda_{k}\}\|_{T_{p}^{q}(Z)}^{s}.
	\end{align*}
	It holds that
	$$\chi_{D(z_k,r)}(z)\lesssim \frac{1-|z_k|}{|1-\overline{z_k}z|} \,, \quad z\in \D\,.$$ 
	In addition,  each $z\in \D$ belongs to not more than  $N=N(r)$  disks $D(z_k,r)$ (see Proposition~\ref{Ncoverdiscs}). Using the fact that 
	$$\|w\|_{t}\leq \max\{1,N^{1/t-1/2}\}\|w\|_{2}$$ for $w\in\C^{N}$ and, combining it with the estimation above, one can see that
	\begin{align*}
	&\int_{\T} \left(\int_{\Gamma_{M_\ast}(\xi)} \sum_k |\lambda_k|^{t}  \frac{\chi_{D(z_k,r)}(z)}{(1-|z_k|)^{\alpha t/p}}\ d\mu(z)\right)^{s/t}\ |d\xi|\\
	&\leq \max\{1,N^{s/t-s/2}\}\,
	\int_{\T} \left(\int_{\Gamma_{M_\ast} (\xi)} \left(\sum_k |\lambda_k|^{2}  \frac{\chi_{D(z_k,r)}(z)}{(1-|z_k|)^{2\alpha/p}}\right)^{t/2}\ d\mu(z)\right)^{s/t}\ |d\xi|\\
	&\lesssim \, \max\{1,N^{s/t-s/2}\}\, \int_{\T} \left(\int_{\Gamma_{M_\ast} (\xi)} \left(\sum_k |\lambda_k|^{2}  \frac{(1-|z_k|)^{2\theta-\frac{2\alpha}{p}}}{|1-\overline{z_k}z|^{2\theta}}\right)^{t/2}\ d\mu(z)\right)^{s/t}\ |d\xi|\\
	&\lesssim \, \max\{1,N^{s/t-s/2}\}\, \|\{\lambda_{k}\}\|_{T_{p}^{q}(Z)}^{s}.
	\end{align*}
	Hence, it follows that 
	\begin{align*}
	\int_{\T} \left( \sum_{z_k\in\Gamma_{M}(\xi)} |\lambda_k|^{t}  \frac{\mu(D(z_k,r))}{(1-|z_k|)^{\alpha t/p}}\ \right)^{s/t}\ |d\xi|\lesssim \max\{1,N^{s/t-s/2}\} \|\{\lambda_{k}\}\|_{T_{p}^{q}(Z)}^{s}.
	\end{align*}
	
	On the other hand, let a $b$  such that $tb> 1$ and $sb> 1$. For any $\{ \tau_k\}_k \in T_{\frac{bt}{bt-1}}^{\frac{bs}{bs-1}}(Z)$, after applying Fubini's
	theorem and Hölder's inequality twice, one gets that
	\begin{align}\label{pairtentcarlesonmeasure}
	\sum_{k} & \tau_k \lambda_{k}^{1/b} \left(\frac{\mu^{1/t}(D(z_k,r))}{(1-|z_k|)^{\alpha/p}}\right)^{1/b} (1-|z_k|)\\\nonumber
	&\asymp \int_{\T} \left(\sum_{z_k\in\Gamma^{-}(\xi)} \tau_k \lambda_{k}^{1/b} \left(\frac{\mu^{1/t}(D(z_k,r))}{(1-|z_k|)^{1/p}}\right)^{1/b}\right) |d\xi|\\
	&\lesssim \int_{\T} \left(\sum_{z_k\in \Gamma^{-}(\xi)} |\lambda_k|^{t}  \frac{\mu(D(z_k,r))}{(1-|z_k|)^{\alpha t/p}}\right)^{1/tb}\left(\sum_{z_k\in \Gamma^{-}(\xi)} \tau_k^{\frac{bt}{bt-1}} \right)^{\frac{bt-1}{bt}}\ |d\xi|\nonumber\\
	&\leq \left(\int_{\T}\left(\sum_{z_k\in \Gamma^{-}(\xi)} |\lambda_k|^{t}  \frac{\mu(D(z_k,r))}{(1-|z_k|)^{\alpha t/p}}\right)^{s/t} |d\xi|\right)^{1/bs} \|\tau_k\|_{T_{\frac{bt}{bt-1}}^{\frac{bs}{bs-1}}(Z)}\nonumber\\
	& \lesssim \max\{1,N^{\frac{s}{tb}-\frac{t}{2b}}\} \|\{\lambda_{k}\}\|_{T_{p}^{q}(Z)}^{1/b} \|\{\tau_k\}\|_{T_{\frac{bt}{bt-1}}^{\frac{bs}{bs-1}}(Z)}.\nonumber
	\end{align}
	
	By Proposition~\ref{FactorizationTentSequence} we have that $$\{\tau_k \lambda_k^{1/b}\}_k\in T_{\frac{bpt}{t+p(bt-1)}}^{\frac{bqs}{s+q(bs-1)}}(Z)=T_{\frac{bt}{bt-1}}^{\frac{bs}{bs-1}}(Z)\cdot T_{bp}^{bq}(Z).$$ Then, using \eqref{pairtentcarlesonmeasure} together with Propositions~\ref{Arsenovic1}--\ref{Luecking1} respectively, we obtain that \emph{(2)} holds. Therefore, we are done.
\end{proof}

 The  last part of the section is about Theorem $3$ of \cite{Luecking}. There, Luecking considered the embeddings of  
 the derivatives of analytic tent spaces $AT^q_{p}(\alpha)$  in the $L^s$ space of a positive Borel measure $\mu$. The setting was the upper half plane and 
 the technique he employed was the  discretization  of the problem, too.
 Nevertheless, the same proof  can be applied to the unit disc.   Below we state  Luecking's  result for the latter case. 
  \begin{theorem}{\cite[Theorem 3, p. 354]{Luecking}}
	Let $0<s,p,q<+\infty$ such that if $s=q$ then $p\leq q$, $0<\alpha<+\infty$, $Z=\{z_{i,j}\}$ the $(r,\kappa)$-lattice consisting of the centers of the Luecking regions, and $n\in\N\cup\{0\}$. Let $\mu$  be a finite positive Borel measure on $\D$. Then the following assertions are equivalent.
	\begin{enumerate}
		\item There is a constant $C>0$ such that
		\begin{align*}
		\left(\int_{\D} |f^{(n)}(w)|^{s}\ d\mu(w)\right)^{1/s}\leq C\|f\|_{T_{p}^{q}(\alpha)}, \quad f\in AT_{p}^{q}(\alpha).
		\end{align*}
		\item The sequence
		\begin{align*}
		\mu_{i,j}:= \mu(R_{i,j})(1-|z_{i,j}|)^{-\frac{s}{p}-sn-1},
		\end{align*}
		where $R_{i,j}$ is a Luecking region (see Example~\ref{LueckingCenters}) and $z_{i,j}$ is the center of this region, satisfies one of the following:
		\begin{enumerate}
			\item If $s<p,q$, then $\{\mu_{i,j}\}\in T_{\frac{p}{p-s}}^{\frac{q}{q-s}}(Z)$.
			\item If $p\leq s<q$, then $\{\mu_{i,j}\}\in T_{\infty}^{\frac{q}{q-s}}(Z)$.
			\item If $q< s$ or $p\leq s=q$ , then $\{\mu_{i,j}(1-|z_{i,j}|)^{1-\frac{s}{q}}\}$ is a bounded sequence. 
		\end{enumerate}
	\end{enumerate}
\end{theorem}
Although the cases stated above as $2(a),2(b),2(c)$ are covered completely by Theorem $3$ in \cite{Luecking},
 looking carefully at the original proof given by Luecking we realize that the case when $0<s=q <p<\infty$
has to be clarified. To be more specific, in  \cite{Luecking} this range of values 
  is also treated  by the application of Proposition $2$ of  \cite{Luecking}. Here we make clear that this case has to be confronted 
  separetely using  the proper duality relation.  When $0<s=q <p$ 
  this duality  is the one stated in Proposition ~\ref{Jevtic}.
  In order to keep the spirit of
  the proof  as it appears in \cite{Luecking} we make use of Luecking's regions.

\begin{theorem}
Let $0<s <p<+\infty$, $0<\alpha<+\infty$, $Z=\{z_{i,j}\}$ the $(r,\kappa)$-lattice consisting of the centers of the Luecking regions, and $n\in\N\cup\{0\}$. Let $\mu$  be a finite positive Borel measure on $\D$. Then, there is a constant $C>0$ such that
	\begin{align}\label{lueckingd}
	\left(\int_{\D} |f^{(n)}(w)|^{s}\ d\mu(w)\right)^{1/s}\leq C\|f\|_{T_{p}^{s}(\alpha)}, \quad f\in AT_{p}^{s}(\alpha),
	\end{align}
 if and only if
 $\left\{\mu(R_{i,j})(1-|z_{i,j}|)^{-\frac{s}{p}-sn-1}\right\}\in T_{\frac{p}{p-s}}^{\infty}(Z)$.
\end{theorem}
\begin{proof}
	First, we prove that the condition $\left\{\mu(R_{i,j})(1-|z_{i,j}|)^{-\frac{s}{p}-sn-1}\right\}\in T_{\frac{p}{p-s}}^{\infty}(Z)$ implies \eqref{lueckingd}. Let $f\in AT^s_p(\alpha)$, then 
	$$|f^{(n)}(z)|\leq \sum_{ij} |f^{(n)}(\tilde{z}_{i,j})|\chi_{R_{i,j}}(z),\quad z\in\D,$$
	where $R_{i,j}$ is a Luecking region (see Example~\ref{LueckingCenters}) and $\tilde{z}_{i,j}\in \overline{R_{i,j}}$ is such that $|f^{(n)}(\tilde{z}_{i,j})| :=\sup_{w\in \overline{R_{i,j}}} |f^{(n)}(w)|$. 
	So that
	\begin{align*}
	&\int_{\D} |f^{(n)}(w)|^{s}\ d\mu(w)\leq \sum_{ij} |f^{(n)}(\tilde{z}_{i,j})|^{s}\ \mu(R_{i,j})\\
	&\quad=\sum_{ij} |f^{(n)}(\tilde{z}_{i,j})|^{s}\ (1-|z_{i,j}|)^{\frac{\alpha s}{p}+sn} \frac{\mu(R_{i,j})}{(1-|z_{i,j}|)^{\frac{\alpha s}{p}+ns+1}}(1-|z_{i,j}|).
	\end{align*}
	Hence, by hypothesis and Proposition~\ref{Jevtic} it follows that
	\begin{align*}
	\int_{\D} |f^{(n)}(w)|^{s}\ d\mu(w)&\leq \sum_{ij} |f^{(n)}(\tilde{z}_{i,j})|^{s}\ (1-|z_{i,j}|)^{\frac{\alpha s}{p}+sn} \frac{\mu(R_{i,j})}{(1-|z_{i,j}|)^{\frac{\alpha s}{p}+ns+1}}(1-|z_{i,j}|)\\
	&\lesssim  \|\{\mu_{i,j}\}\|_{T_{\frac{p}{p-s}}^{\infty}(Z)} \|\{|f^{(n)}(\tilde{z}_{i,j})|^{s}\ (1-|z_{i,j}|)^{\frac{\alpha s}{p}+sn}\}\|_{T_{p/s}^{1}(Z)}.
	\end{align*}
	
	Let us check that $$\| \{|f^{(n)}(\tilde{z}_{i,j})|^{s}(1-|z_k|)^{\frac{\alpha s}{p}+sn}\}\|_{T_{p/s}^{1}(Z)} \lesssim \|f\|_{T_{p}^{s}(\alpha)}.$$ 
	
	Using the fact that there is $\eta>0$ such that $\overline{R_{i,j}}\subset D(z_{i,j},\eta)$. Applying the mean property over each hyperbolic disc $D(\tilde{z}_{i,j},\eta)$ (see Lemma~\ref{MVP}), it follows that
	\begin{align*}
	&\|\{|f^{(n)}(\tilde{z}_{i,j})|^{s}\ (1-|z_{i,j}|)^{\frac{\alpha s}{p}+sn}\}\|_{T_{p/s}^{1}(Z)}\\
	&\quad=\int_{\T}\left(\sum_{z_{i,j}\in \Gamma_{M}(\xi)} |f^{(n)}(\tilde{z}_{i,j})|^{p}\ (1-|z_{i,j}|)^{\alpha+pn}\right)^{s/p}\ |d\xi| \\
	&\quad \lesssim \int_{\T}\left(\sum_{z_{i,j}\in \Gamma_{M}(\xi)} \frac{(1-|z_{i,j}|)^{pn+\alpha}}{(1-|\tilde{z}_{i,j}|)^{np+2}}\int_{D(\tilde{z}_{i,j},\eta)}|f(w)|^p\ dm(w)\ \right)^{s/p}\ |d\xi| .
	\end{align*}
	Since $D(\tilde{z}_{i,j},\eta)\subset D(z_{i,j},2\eta)$, by Remark~\ref{remarkestimatedistbound} we have 
	\begin{align*}
	&\|\{|f^{(n)}(\tilde{z}_{i,j})|^{s}\ (1-|z_{i,j}|)^{\frac{\alpha s}{p}+sn}\}\|_{T_{p/s}^{1}(Z)}\\
	&\quad \lesssim \int_{\T}\left(\sum_{z_{i,j}\in \Gamma_{M}(\xi)} \int_{D(z_{i,j},2\eta)}|f(w)|^p\ \frac{dm(w)}{(1-|w|)^{2-\alpha}}\ \right)^{s/p}\ |d\xi| .
	\end{align*}
	
	By Lemma~\ref{Wu}, we can take $M_{+}>M>1/2$ such that
	$$\bigcup_{z_{i,j}\in\Gamma_{M}(\xi)} D(z_{i,j},2\eta)\subset \Gamma_{M_{+}}(\xi),$$
	we have
	\begin{align*}
&\|\{|f^{(n)}(\tilde{z}_{i,j})|^{s}\ (1-|z_{i,j}|)^{\frac{\alpha s}{p}+sn}\}\|_{T_{p/s}^{1}(Z)}\\
&\quad\lesssim \int_{\T}\left(\int_{\Gamma_{M_+}(\xi)} \left(\sum_{z_{i,j}\in \Gamma_{M}(\xi)} \chi_{D(z_{i,j},2\eta)}(w)\right)|f(w)|^p\ \frac{dm(w)}{(1-|w|)^{2-\alpha}}\ \right)^{s/p}\ |d\xi|.
	\end{align*}
	
Since each $z\in \D$ can be covered by a fixed number $N$ of hyperbolic discs $D(z_{i,j},2\eta)$ (see Proposition~\ref{Ncoverdiscs}), we obtain 
$$ \|\{|f^{(n)}(\tilde{z}_{i,j})|^{s}\ (1-|z_{i,j}|)^{\frac{\alpha s}{p}+sn}\}\|_{T_{p/s}^{1}(Z)}\lesssim \|f\|_{T_{p}^{s}(\alpha)}^{s}.
$$	


	Conversely, assume that there is a constant $C>0$ such that 
	\begin{align*}
	\left(\int_{\D} |f^{(n)}(w)|^{s}\ d\mu(w)\right)^{1/s}\leq C \|f\|_{T_{p}^{s}(\alpha)}
	\end{align*}
	for every $f\in AT_{p}^{s}(\alpha)$. Let us see that  $\left\{\mu(R_{i,j})(1-|z_{i,j}|)^{-\frac{s}{p}-sn-1}\right\}\in T_{\frac{p}{p-s}}^{\infty}(Z)$. Let $\lambda=\{\lambda_{i,j}\}$ be a sequence of $T_{p}^{s}(Z)$ and $\rho_{i,j}: [0,1]\rightarrow \{-1,1\}$ the Radermacher functions.
	Moreover, for $M>\max\left\{1,p/s,1/s,1/p\right\}+1/p$ we consider the function
	\begin{align*}
	F_x(z):=\sum_{ij} \lambda_{i,j} \rho_{i,j}(x) \frac{(1-|z_{i,j}|)^{M-\frac{1}{p}}}{(1-\overline{z_{i,j}}z)^{M}}\,, \quad z \in \mathbb D\,.
	\end{align*}
	Using Proposition~\ref{Soperator}, we have that
	\begin{align*}
	\left(\int_{\D} |F_{x}^{(n)}(w)|^{s}\ d\mu(w)\right)^{1/s}\leq C \|F_x\|_{T_{p}^{s}(\alpha)}\lesssim \|\lambda\|_{T_{p}^{s}(Z)}. 
	\end{align*}
	Integrating both sides with respect to $x$ we obtain
	\begin{align*}
	\int_{\D}\int_{0}^{1} |F_{x}^{(n)}(w)|^{s}\ dx\ d\mu(w)\lesssim \|\lambda\|_{T_{p}^{s}(Z)}^{s}. 
	\end{align*}
	By means of the Khinchine's inequality we get that 
	\begin{align*}
	\int_{\D} \left(\sum_{i,j} |\lambda_{i,j}|^{2} |z_{i,j}|^{2n} \frac{(1-|z_{i,j}|)^{2M-\frac{2}{p}}}{|1-z\overline{z_{i,j}}|^{2M+2n}}\right)^{s/2}\ d\mu(z)\lesssim \|\lambda\|_{T_{p}^{s}(Z)}^{s}. 
	\end{align*}
	Since $\chi_{R_{ij}}(z)\lesssim \frac{1-|z_{i,j}|}{|1-\overline{z_{i,j}}z|}$, we have that
	\begin{align*}
	\int_{\D} \left(\sum_{ij} |\lambda_{i,j}|^{2}  \frac{\chi_{R_{i,j}}(z)}{(1-|z_{i,j}|)^{2n+\frac{2}{p}}}\right)^{s/2}\ d\mu(z)\lesssim \|\lambda\|_{T_{p}^{s}(Z)}^{s}. 
	\end{align*}
	Therefore, it follows that
	\begin{align*}
	\sum_{ij} |\lambda_{i,j}|^{s} \frac{\mu(R_{i,j})}{(1-|z_{i,j}|)^{ns+\frac{s}{p}+1}}\ (1-|z_{i,j}|)\lesssim \|\lambda\|_{T_{p}^{s}(Z)}^{s}=\|\{|\lambda_{k}|^{s}\}\|_{T_{p/s}^{1}(Z)}. 
	\end{align*}
	
	The arbitrariness of $\{|\lambda_{i,j}|^{s}\}$ in $T_{p/s}^{1}(Z)$ and Proposition~\ref{Jevtic} show that $$\left\{\mu(R_{i,j})(1-|z_{i,j}|)^{-\frac{s}{p}-sn-1}\right\}\in T_{\frac{p}{p-s}}^{\infty}(Z).$$

\end{proof}

\section{Integration operator}
In this last section, we focus on the application of our main result to the study of the integration operator 
 when acting on the $RM(p,q)$ spaces.
 Let a $g\in\mathcal{H}(\D)$, by the term intergration operator we refer to the transformation 
$$
T_{g}(f)(z)=\int_{0}^{z} f(\zeta)g'(\zeta)\ d\zeta\,,  \quad z\in \mathbb D\,,
$$
where $ f \in \mathcal H (\mathbb D)\,$.

We recall that for {$1\leq p< +\infty$, $1\leq q\leq +\infty$}
$$T_g : RM(p,q)\rightarrow RM(p,q)$$
bounded  if and only if $g\in\mathcal{B}$ (see \cite{Aguilar-Contreras-Piazza_2}).  The same condition characterizes the boundedness of the operator 
on Bergman spaces (see \cite{AS2}).

The starting point for the study of $T_g$ was the  Hardy space setting. It was proved that,  
$$T_{g}: H^{q}\rightarrow H^q $$
 is bounded if and only if $g\in BMOA$ (see \cite{Pom},\cite{AS}). 
 Since then many authors have considered the action of $T_g$ between distinct Bergman, Hardy  spaces (see, i.e., \cite{Wu_2011},\cite{MiihkinenPauPeralaWang2020}).
 Based on the fact that  an $RM(p,q)$ space is identified as a Bergman, when $p=q$, and a Hardy space $H^q$ corresponds to the limit case $RM(\infty,q)$, here  we consider the more general question of characterizing the symbols $g$ such that 
 $$T_g: RM(p,q)\rightarrow RM(t,s)$$ 
 for $1\leq p,q ,t,s \leq +\infty$.  It turns out that Theorem $4.1$ is the key to this study.
 
First we present the  answer  to the question under consideration when the indices are finite.
\begin{theorem}\label{RMRM}
	Let $1\leq p,q,s,t<+\infty$. The following statements are equivalent:
	\begin{enumerate}
		\item The operator $T_{g}:RM(p,q)\rightarrow RM(t,s)$ is bounded.
		\item If $Z=\{z_k\}$ is an $(r,\kappa)$-lattice and denoting $d\mu_g(z):=|g'(z)|^t\ dm(z)$ it holds
		\begin{enumerate}
			\item If $1\leq s<q<+\infty$, $1\leq t<p<+\infty$, then
			\begin{align*}
			\int_{\T} \left(\sum_{z_k\in \Gamma(\xi)} \left(\frac{\mu_{g}^{1/t}(D(z_k,r))}{(1-|z_k|)^{\frac{1}{p}+\frac{1}{t}-1}}\right)^{\frac{pt}{p-t}} \right)^{\frac{(p-t)qs}{(q-s)pt}}\ |d\xi|<+\infty.
			\end{align*}
			\item If $1\leq s< q<+\infty$, $1\leq p\leq t<+\infty$, then
			\begin{align*}
			\int_{\T} \left(\sup_{z_k\in \Gamma(\xi)} \frac{\mu_{g}^{1/t}(D(z_k,r))}{(1-|z_k|)^{\frac{1}{p}+\frac{1}{t}-1}} \right)^{\frac{qs}{q-s}}\ |d\xi|<+\infty.
			\end{align*}
			\item If $1\leq q< s<+\infty$, $1\leq p,t<\infty$ or $1\leq q=s<+\infty$, $1\leq p\leq t<+\infty$, then
			\begin{align*}
			\sup_{k} \frac{\mu_{g}^{1/t}(D(z_k,r))}{(1-|z_k|)^{\frac{1}{p}+\frac{1}{q}+\frac{1}{t}-\frac{1}{s}-1}}<+\infty.
			\end{align*}
			\item If $1\leq q=s<+\infty$, $1\leq t< p<+\infty$, then
			\begin{align*}
			\sup_{\xi\in\T} \left(\sup_{\xi\in I}  \sum_{z_{k}\in S(I)} \left(\frac{\mu_{g}^{1/t}(D(z_k,r))}{(1-|z_k|)^{\frac{1}{p}+\frac{1}{t}-1}}\right)^{\frac{pt}{p-t}} (1-|z_k|) \right)^{\frac{p-t}{pt}}<+\infty
			\end{align*}
		\end{enumerate}
		\item The function $g\in \mathcal{H}(\D)$ satisfies that
		\begin{enumerate}
			\item If $1\leq s<q<+\infty$, $1\leq t<p<+\infty$, 
			\begin{align*}
			g\in RM\left(\frac{pt}{p-t},\frac{qs}{q-s}\right).
			\end{align*}
			\item If $1\leq s< q<+\infty$, $1\leq p\leq t<+\infty$, 
			\begin{align*}
			g'(z)(1-|z|)^{1+\frac{1}{t}-\frac{1}{p}}\in T_{\infty}^{\frac{qs}{q-s}}.
			\end{align*}
			\item If $1\leq q< s<+\infty$, $1\leq p,t<\infty$ or $1\leq q=s<+\infty$, $1\leq p\leq t<+\infty$, 
			\begin{align*}
			g\in\mathcal{B}^{1+\frac{1}{t}+\frac{1}{s}-\frac{1}{p}-\frac{1}{q}}.
			\end{align*}
			\item If $1\leq q=s<+\infty$, $1\leq t< p<+\infty$, the measure $$d\mu(z)=|g'(z)|^{\frac{pt}{p-t}}(1-|z|)^{\frac{pt}{p-t}}\ dm(z)$$ is a Carleson measure.
		\end{enumerate}
	\end{enumerate}
\end{theorem}
\begin{proof}
Fix $M>1/2$. In the begining of the proof we remind that, given $f\in \mathcal{H}(\D)$,
\begin{align*}
	\rho_{t,s}(T_{g}(f))& \asymp\left(\int_{\T} \left(\int_{\Gamma_{M}(\xi)} |T_g(f)'(w)|^{t} (1-|w|)^{t-1}\ dm(w)\right)^{s/t}\ |d\xi|\right)^{1/s}\\
	& = \left(\int_{\T} \left(\int_{\Gamma_{M}(\xi)} |f(w)|^{t} |g'(w)|^{t} (1-|w|)^{t-1}\ dm(w)\right)^{s/t}\ |d\xi|\right)^{1/s}
	\end{align*}
	by Theorems~\ref{non_tangential_char_without_der}--\ref{characterization 2}.  Now, consider $\mu $ the  measure
	$$
	d\mu(z)=|g'(w)|^{t}(1-|w|)^{t-1}\ dm(w)=(1-|w|)^{t-1}\ d\mu_g(w)\,.
	$$ 
	In addition,  taking into account that given the $(r,\kappa)$-lattice $\{z_k\}$ ,
	$$ 
	1-|w| \asymp  1-|z_k|\,,\ w\in\D(z,k,r),
	$$
	for any  $w $ in the hyperbolic disc $ D(z_k, r)$, and we have 
	$$\mu(D(z_k,r))\asymp (1-|z_k|)^{t-1} \mu_{g}(D(z_k,r)).$$
	Bearing in mind these facts, the equivalences between \emph{(1)} and \emph{(2)} follow immediately by using Theorem~\ref{theoremCarlesonTent}.
		

	From now on, we prove the equivalence between \emph{(2)} and \emph{(3)}. We split the proof in different cases.
	
	Let us prove that \emph{(2)(a)} is equivalent to \emph{(3)(a)}. The idea is to show that 
	\begin{align}\label{equivalencia(a)}
		&\int_{\T} \left(\sum_{z_k\in \Gamma_{M}(\xi)} \left(\frac{\mu_{g}^{1/t}(D(z_k,r))}{(1-|z_k|)^{\frac{1}{p}+\frac{1}{t}-1}}\right)^{\frac{pt}{p-t}} \right)^{\frac{(p-t)qs}{(q-s)pt}}\ |d\xi|\\\nonumber
		&\quad\asymp \int_{\T} \left(\int_{\Gamma_{M}(\xi)} {|g'(w)|^{\frac{pt}{p-t}}(1-|w|)^{\frac{pt}{p-t}-1}}\ dm(w) \right)^{\frac{(p-t)qs}{(q-s)pt}}\ |d\xi|.
	\end{align}
	First, let us check that 
	\begin{align}\label{eq(a)}
	&\int_{\T} \left(\sum_{z_k\in \Gamma_{M}(\xi)} \left(\frac{\mu_{g}^{1/t}(D(z_k,r))}{(1-|z_k|)^{\frac{1}{p}+\frac{1}{t}-1}}\right)^{\frac{pt}{p-t}} \right)^{\frac{(p-t)qs}{(q-s)pt}}\ |d\xi|\\\nonumber
	&\quad\lesssim \int_{\T} \left(\int_{\Gamma_{M}(\xi)} {|g'(w)|^{\frac{pt}{p-t}}(1-|w|)^{\frac{pt}{p-t}-1}}\ dm(w) \right)^{\frac{(p-t)qs}{(q-s)pt}}\ |d\xi|.
	\end{align}
	Setting $|g'(\tilde{z}_k)|:=\sup_{w\in \overline{D(z_k,r)}} |g'(w)|$ we have
	\begin{align*}
	&\int_{\T} \left(\sum_{z_k\in \Gamma_{M}(\xi)} \left(\frac{\mu_{g}^{1/t}(D(z_k,r))}{(1-|z_k|)^{\frac{1}{p}+\frac{1}{t}-1}}\right)^{\frac{pt}{p-t}} \right)^{\frac{(p-t)qs}{(q-s)pt}}\ |d\xi|\\
	&\quad\lesssim\int_{\T} \left(\sum_{z_k\in \Gamma_{M}(\xi)}   \left( \sup_{w\in \overline{D(z_k,r)}} |g'(w)| (1-|z_k|)^{\frac{1}{t}-\frac{1}{p}+1}
	\right)^{\frac{pt}{p-t}} \right)^{\frac{(p-t)qs}{(q-s)pt}}\ |d\xi|\\
	& \quad = \int_{\T} \left(\sum_{z_k\in \Gamma_{M}(\xi)}   \left( |g'(\tilde{z}_k)| (1-|z_k|)^{\frac{1}{t}-\frac{1}{p}+1}
	\right)^{\frac{pt}{p-t}} \right)^{\frac{(p-t)qs}{(q-s)pt}}\ |d\xi|\,.
	\end{align*}
	At this point we apply the mean value property over the hyperbolic disc $D(\tilde{z}_k,s)$ with $s<3r$ (see Lemma~\ref{MVP}), so
	\begin{align*}
	&\int_{\T} \left(\sum_{z_k\in \Gamma_{M}(\xi)} \left(\frac{\mu_{g}^{1/t}(D(z_k,r))}{(1-|z_k|)^{\frac{1}{p}+\frac{1}{t}-1}}\right)^{\frac{pt}{p-t}} \right)^{\frac{(p-t)qs}{(q-s)pt}}\ |d\xi|\\
	&\quad\lesssim  \int_{\T} \left(\sum_{z_k\in \Gamma_{M}(\xi)} \frac{(1-|z_k|)^{1+\frac{pt}{p-t}}}{(1-|\tilde{z}_k|)^{2}}  \left( \int_{D(\tilde{z}_k,s)}|g'(w)|^{\frac{pt}{p-t}} \ dm(w)
	\right) \right)^{\frac{(p-t)qs}{(q-s)pt}}\ |d\xi|.
	\end{align*}
	Since $D(\tilde{z}_k,s)\subset D(z_k,4r)$, by Remark~\ref{remarkestimatedistbound}, we have
	\begin{align*}
	&\int_{\T} \left(\sum_{z_k\in \Gamma_{M}(\xi)} \left(\frac{\mu_{g}^{1/t}(D(z_k,r))}{(1-|z_k|)^{\frac{1}{p}+\frac{1}{t}-1}}\right)^{\frac{pt}{p-t}} \right)^{\frac{(p-t)qs}{(q-s)pt}}\ |d\xi|\\
	&\quad\lesssim  \int_{\T} \left(\sum_{z_k\in \Gamma_{M}(\xi)}  \left( \int_{D({z}_k,4r)}|g'(w)|^{\frac{pt}{p-t}} (1-|w|)^{\frac{pt}{p-t}-1} \ dm(w)
	\right) \right)^{\frac{(p-t)qs}{(q-s)pt}}\ |d\xi|.
	\end{align*}
	Taking $M_{+}>M>1/2$ such that
	$$
	\bigcup_{D(z_{k},r)\cap \Gamma_{M} (\xi)\neq \emptyset } D(z_k,4r) \subset \Gamma_{M_{+}}(\xi),
	$$
	we get that
	\begin{align*}
	&\int_{\T} \left(\sum_{z_k\in \Gamma_{M}(\xi)} \left(\frac{\mu_{g}^{1/t}(D(z_k,r))}{(1-|z_k|)^{\frac{1}{p}+\frac{1}{t}-1}}\right)^{\frac{pt}{p-t}} \right)^{\frac{(p-t)qs}{(q-s)pt}}\ |d\xi|\\
	&\quad\lesssim \int_{\T} \left(\int_{\Gamma_{M_{+}}(\xi)} \left(\sum_{z_k\in \Gamma_{M}(\xi)}\chi_{D(z_k,4r)}(w) \right) \frac{|g'(w)|^{\frac{pt}{p-t}}}{(1-|w|)^{1-\frac{pt}{p-t}}} \ dm(w) \right)^{\frac{(p-t)qs}{(q-s)pt}}\ |d\xi|.
	\end{align*}
	According to Proposition~\ref{Ncoverdiscs}, each $z\in \D$ belongs to no more than $N$ hyperbolic discs $D(z_k,4r)$ with $N$ only depending on the lattice. As a consequence  
	\begin{align*}
	&\int_{\T} \left(\sum_{z_k\in \Gamma_{M}(\xi)} \left(\frac{\mu_{g}^{1/t}(D(z_k,r))}{(1-|z_k|)^{\frac{1}{p}+\frac{1}{t}-1}}\right)^{\frac{pt}{p-t}} \right)^{\frac{(p-t)qs}{(q-s)pt}}\ |d\xi|\\
	&\quad\lesssim \int_{\T} \left(\int_{\Gamma_{M_{+}}(\xi)}  |g'(w)|^{\frac{pt}{p-t}}(1-|w|)^{\frac{pt}{p-t}-1} \ dm(w) \right)^{\frac{(p-t)qs}{(q-s)pt}}\ |d\xi|\\
	&\quad\lesssim \int_{\T} \left(\int_{\Gamma_{M}(\xi)}  |g'(w)|^{\frac{pt}{p-t}}(1-|w|)^{\frac{pt}{p-t}-1} \ dm(w) \right)^{\frac{(p-t)qs}{(q-s)pt}}\ |d\xi|,
	\end{align*}
	where the last inequality follows by Lemma~\ref{estimate 1}. So that, \eqref{eq(a)} holds.
	
	Now, we proceed with the converse inequality. Using the pointwise estimate $|g'(z)|\leq \sum_{k} |g'(z)| \chi_{D({z}_k,r)}(z)$ and the fact (given by Lemma~\ref{Wu}) that we can take $M_{+}>M>1/2$ such that
	$$
	\bigcup_{D(z_{k},r)\cap \Gamma_{M}(\xi)\neq \emptyset } D(z_k,r)\subset \Gamma_{M_+}(\xi),
	$$
	it follows that
	\begin{align*}
	&\int_{\T}\left(\int_{\Gamma_{M}(\xi)} |g'(w)|^{\frac{pt}{p-t}}(1-|w|)^{\frac{pt}{p-t}-1} {\ dm(w)}{}\right)^{\frac{(p-t)qs}{(q-s)pt}} |d\xi|\\
	&\quad\leq\int_{\T}\left(\sum_{k}\int_{\Gamma_{M}(\xi)\cap D({z}_k,r)} |g'(w)|^{\frac{pt}{p-t}}(1-|w|)^{\frac{pt}{p-t}-1} \ dm(w)\right)^{\frac{(p-t)qs}{(q-s)pt}} |d\xi|\\
	&\quad \leq \int_{\T}\left(\sum_{z_{k}\in \Gamma_{M_+}(\xi)}\int_{D({z}_k,r)} |g'(w)|^{\frac{pt}{p-t}}(1-|w|)^{\frac{pt}{p-t}-1} \ dm(w)\right)^{\frac{(p-t)qs}{(q-s)pt}} |d\xi|.
	\end{align*}
	Taking $\tilde{z}_{k}\in\overline{D(z_k,r)}$ such that $|g'(\tilde{z}_{k})|:=\sup_{w\in\overline{D(z_k,r)}} |g'(w)|$, and using Remark~\ref{remarkestimatedistbound}, we have
	\begin{align*}
	&\int_{\T}\left(\int_{\Gamma_{M}(\xi)} |g'(w)|^{\frac{pt}{p-t}}(1-|w|)^{\frac{pt}{p-t}-1} {\ dm(w)}{}\right)^{\frac{(p-t)qs}{(q-s)pt}} |d\xi|\\
	&\quad \lesssim \int_{\T}\left(\sum_{z_{k}\in \Gamma_{M_+}(\xi)} |g'(\tilde{z}_{k})|^{\frac{pt}{p-t}}\ (1-|z_k|)^{\frac{pt}{p-t}+1} \right)^{\frac{(p-t)qs}{(q-s)pt}} |d\xi|.
	\end{align*}
	By the mean value property over the hyperbolic disc $D(\tilde{z}_{k},s)$ with $s<3r$ (see Lemma~\ref{MVP}), we obtain 
	\begin{align*}
	&\int_{\T}\left(\int_{\Gamma_{M}(\xi)} |g'(w)|^{\frac{pt}{p-t}}(1-|w|)^{\frac{pt}{p-t}-1} {\ dm(w)}{}\right)^{\frac{(p-t)qs}{(q-s)pt}} |d\xi|\\
	&\quad \lesssim \int_{\T}\left(\sum_{z_{k}\in \Gamma_{M_+}(\xi)} \left(\int_{D(\tilde{z}_{k},s)} |g'(w)|^{t}\ \frac{dm(w)}{(1-|\tilde{z}_{k}|)^{2}}\right)^{\frac{p}{p-t}}\ (1-|z_k|)^{\frac{pt}{p-t}+1} \right)^{\frac{(p-t)qs}{(q-s)pt}} |d\xi|.
	\end{align*}
	Since $D(\tilde{z}_{k},s)\subset D(z_k,4r)$ and Remark~\ref{remarkestimatedistbound}, we have
	\begin{align*}
	&\int_{\T}\left(\int_{\Gamma_{M}(\xi)} |g'(w)|^{\frac{pt}{p-t}}(1-|w|)^{\frac{pt}{p-t}-1} {\ dm(w)}{}\right)^{\frac{(p-t)qs}{(q-s)pt}} |d\xi|\\
	&\quad \lesssim \int_{\T}\left(\sum_{z_{k}\in \Gamma_{M_+}(\xi)} \left(\int_{D({z_k},4r)} |g'(w)|^{t}\ \ dm(w)\right)^{\frac{p}{p-t}}\ (1-|z_k|)^{\frac{pt-p-t}{p-t}} \right)^{\frac{(p-t)qs}{(q-s)pt}} |d\xi|\\
	&\quad \lesssim \int_{\T}\left(\sum_{z_{k}\in \Gamma_{M_+}(\xi)} \left(\sum_{j}\frac{\int_{D({z_j},r)} \chi_{D(z_k,4r)}(w) |g'(w)|^{t}\ \ dm(w)}{(1-|z_k|)^{1+\frac{t}{p}-t}}\right)^{\frac{p}{p-t}}\  \right)^{\frac{(p-t)qs}{(q-s)pt}} |d\xi|.
	\end{align*}
	Applying Remark~\ref{remarkestimatedistbound} one more time, it follows that
	\begin{align*}
	&\int_{\T}\left(\int_{\Gamma_{M}(\xi)} |g'(w)|^{\frac{pt}{p-t}}(1-|w|)^{\frac{pt}{p-t}-1} {\ dm(w)}\right)^{\frac{(p-t)qs}{(q-s)pt}} |d\xi|\\
	&\quad \lesssim \int_{\T}\left(\sum_{z_{k}\in \Gamma_{M_+}(\xi)} \left(\sum_{z_{j}\in D(z_k,5r)}\frac{\mu_{g}^{1/t}(D(z_j,r))}{(1-|z_j|)^{\frac{1}{t}+\frac{1}{p}-1}}\right)^{\frac{pt}{p-t}}\  \right)^{\frac{(p-t)qs}{(q-s)pt}} |d\xi|.
	\end{align*}
	By Proposition~\ref{Ncoverdiscs}, there are at most $N=N(r,\kappa)$  points $z_j$ of the $(r,\kappa)$-lattice in the disc $D(z_k,5r)$. Using the fact that 
	$$\|w\|_{\ell^{1}}\leq N^{t/p} \|w\|_{\ell^{\frac{p}{p-t}}}$$
	for all $w\in\C^{N}$, we obtain 
	\begin{align*}
	&\int_{\T}\left(\int_{\Gamma_{M}(\xi)} |g'(w)|^{\frac{pt}{p-t}}(1-|w|)^{\frac{pt}{p-t}-1} {\ dm(w)}\right)^{\frac{(p-t)qs}{(q-s)pt}} |d\xi|\\
	&\quad \lesssim \int_{\T}\left(\sum_{z_{k}\in \Gamma_{M_+}(\xi)} \sum_{z_{j}\in D(z_k,5r)} N^{t/p}\left(\frac{\mu_{g}^{1/t}(D(z_j,r))}{(1-|z_j|)^{\frac{1}{t}+\frac{1}{p}-1}}\right)^{\frac{pt}{p-t}}\  \right)^{\frac{(p-t)qs}{(q-s)pt}} |d\xi|.
	\end{align*}
	Now, by Lemma~\ref{Wu}, one can take $M_{\ast}>M_{+}>1/2$ such that
	\begin{align*}
	\bigcup_{z_{k}\in \Gamma_{M_{+}}(\xi)} D(z_k,5r)\subset \Gamma_{M_\ast}(\xi).
	\end{align*}
	Hence, we have
	\begin{align*}
	&\int_{\T}\left(\int_{\Gamma_{M}(\xi)} |g'(w)|^{\frac{pt}{p-t}}(1-|w|)^{\frac{pt}{p-t}-1} {\ dm(w)}\right)^{\frac{(p-t)qs}{(q-s)pt}} |d\xi|\\
	&\quad \lesssim N^{\frac{(p-t)qs}{(q-s)p^{2}}} \int_{\T}\left(\sum_{j}\sum_{z_{k}\in \Gamma_{M_+}(\xi)} \chi_{D(z_j,5r)}(z_k) \left(\frac{\mu_{g}^{1/t}(D(z_j,r))}{(1-|z_j|)^{\frac{1}{t}+\frac{1}{p}-1}}\right)^{\frac{pt}{p-t}}\  \right)^{\frac{(p-t)qs}{(q-s)pt}} |d\xi|\\
	&\quad \lesssim N^{\frac{(p-t)qs}{(q-s)p^{2}}} \int_{\T}\left(\sum_{z_{j}\in \Gamma_{M_\ast}(\xi) }\left(\frac{\mu_{g}^{1/t}(D(z_j,r))}{(1-|z_j|)^{\frac{1}{t}+\frac{1}{p}-1}}\right)^{\frac{pt}{p-t}}\sum_{z_{k}\in \Gamma_{M_+}(\xi)} \chi_{D(z_j,5r)}(z_k) \  \right)^{\frac{(p-t)qs}{(q-s)pt}} |d\xi|.
	\end{align*}
	Since, for all $j\in\N$,
	$$\sum_{z_{k}\in \Gamma_{M_+}(\xi)} \chi_{D(z_j,5r)}(z_k)\leq N,$$
	where $N=N(r,\kappa)$, we have
	\begin{align*}
	&\int_{\T}\left(\int_{\Gamma_{M}(\xi)} |g'(w)|^{\frac{pt}{p-t}}(1-|w|)^{\frac{pt}{p-t}-1} {\ dm(w)}\right)^{\frac{(p-t)qs}{(q-s)pt}} |d\xi|\\
	&\quad \lesssim N^{\frac{(p-t)qs}{(q-s)p^{2}}} N^{\frac{(p-t)qs}{(q-s)pt}} \int_{\T}\left(\sum_{z_{j}\in \Gamma_{M_\ast}(\xi) }\left(\frac{\mu_{g}^{1/t}(D(z_j,r))}{(1-|z_j|)^{\frac{1}{t}+\frac{1}{p}-1}}\right)^{\frac{pt}{p-t}}\  \right)^{\frac{(p-t)qs}{(q-s)pt}} |d\xi|
	\\&\quad\lesssim\int_{\T}\left(\sum_{z_{j}\in \Gamma_{M}(\xi) }\left(\frac{\mu_{g}^{1/t}(D(z_j,r))}{(1-|z_j|)^{\frac{1}{t}+\frac{1}{p}-1}}\right)^{\frac{pt}{p-t}}\  \right)^{\frac{(p-t)qs}{(q-s)pt}} |d\xi|.
	\end{align*}
	The last inequality follows by Lemma~\ref{estimate 1}. Therefore \eqref{equivalencia(a)} holds.
	
	Combining Theorem~\ref{characterization 2} and Theorem~\ref{non_tangential_char_without_der}, it follows that
	\begin{align*}
	\rho_{\frac{pt}{p-t},\frac{qs}{q-s}}(g)\asymp \|g\|_{T_{\frac{pt}{p-t}}^{\frac{qs}{q-s}}(1)} \asymp \|g'(z)(1-|z|)\|_{T_{\frac{pt}{p-t}}^{\frac{qs}{q-s}}(1)}.
	\end{align*}
	By \eqref{equivalencia(a)} and Remark~\ref{independetregion}, we conclude
	\begin{align*}
	\rho_{\frac{pt}{p-t},\frac{qs}{q-s}}(g)\asymp\left(\int_{\T} \left(\sum_{z_k\in \Gamma_{M}(\xi)} \left(\frac{\mu_{g}^{1/t}(D(z_k,r))}{(1-|z_k|)^{\frac{1}{p}+\frac{1}{t}-1}}\right)^{\frac{pt}{p-t}} \right)^{\frac{(p-t)qs}{(q-s)pt}}\ |d\xi|\right)^{\frac{q-s}{qs}}.
	\end{align*}
	So that, we have proved that \emph{(2)(a)} and \emph{(3)(a)} are equivalent.
	\newline

	Let us show the equivalence between \emph{(2)(b)} and \emph{(3)(b)}, that is,
	\begin{align}\label{2b1}
	&\int_{\T} \left(\sup_{w\in \Gamma_{M}(\xi)} |g'(w)|(1-|w|)^{1+\frac{1}{t}-\frac{1}{p}} \right)^{\frac{qs}{q-s}}\ |d\xi|\\\nonumber
	&\quad \asymp 	\int_{\T} \left(\sup_{z_k\in \Gamma_{M}(\xi)} \frac{\mu_{g}^{1/t}(D(z_k,r))}{(1-|z_k|)^{\frac{1}{p}+\frac{1}{t}-1}} \right)^{\frac{qs}{q-s}}\ |d\xi|.\nonumber
	\end{align}
	
	First, taking supremum over each hyperbolic disc $D(z_k,r)$ and  $M_{+}>M>1/2$ such that
	$$
	\bigcup_{D(z_{k},r)\cap \Gamma_{M} (\xi)\neq \emptyset } D(z_k,r) \subset  \Gamma_{M_{+}}(\xi),
	$$
	it follows that
	\begin{align*}
	&\int_{\T} \left(\sup_{z_k\in \Gamma_{M}(\xi)} \frac{\mu_{g}^{1/t}(D(z_k,r))}{(1-|z_k|)^{\frac{1}{p}+\frac{1}{t}-1}} \right)^{\frac{qs}{q-s}}\ |d\xi| \\
	&\quad\leq \int_{\T} \left(\sup_{w\in \Gamma_{M_{+}}(\xi)} |g'(w)|(1-|w|)^{1+\frac{1}{t}-\frac{1}{p}} \right)^{\frac{qs}{q-s}}\ |d\xi|\\
	&\quad\leq \int_{\T} \left(\sup_{w\in \Gamma_{M}(\xi)} |g'(w)|(1-|w|)^{1+\frac{1}{t}-\frac{1}{p}} \right)^{\frac{qs}{q-s}}\ |d\xi|,
	\end{align*}
	where in the last inequality we use Lemma~\ref{lemma_sup} applied to a dense sequence in $\D$. So, one direction of \eqref{2b1} holds.
	
	Now we continue with the proof of the converse inequality. As before, we can choose $M_{+}>M>1/2$ such that
	$$
	\bigcup_{D(z_{k},r)\cap \Gamma_{M}(\xi)\neq \emptyset } D(z_k,r)\subset \Gamma_{M_+}(\xi).
	$$
	So that, we have
	\begin{align*}
	&\int_{\T} \left(\sup_{w\in \Gamma_{M}(\xi)} |g'(w)|(1-|w|)^{1+\frac{1}{t}-\frac{1}{p}} \right)^{\frac{qs}{q-s}}\ |d\xi|\\
	&\quad \leq \int_{\T} \left(\sup_{z_{k}\in \Gamma_{M_+}(\xi)}\sup_{w\in D(z_k,r)} |g'(w)|(1-|w|)^{1+\frac{1}{t}-\frac{1}{p}} \right)^{\frac{qs}{q-s}}\ |d\xi|.
	\end{align*}
	Taking $\tilde{z}_{k}\in \overline{D(z_k,r)}$ such that $|g'(\tilde{z}_{k})|:=\sup_{w\in\overline{D(z_k,r)}} |g'(w)|$, applying Remark~\ref{remarkestimatedistbound} and Lemma~\ref{MVP} for $s<3r$ we obtain
	\begin{align*}
	&\int_{\T} \left(\sup_{w\in \Gamma_{M}(\xi)} |g'(w)|(1-|w|)^{1+\frac{1}{t}-\frac{1}{p}} \right)^{\frac{qs}{q-s}}\ |d\xi|\\
	&\quad \lesssim \int_{\T} \left(\sup_{z_{k}\in \Gamma_{M_+}(\xi)} |g'(\tilde{z}_{k})|(1-|z_k|)^{1+\frac{1}{t}-\frac{1}{p}} \right)^{\frac{qs}{q-s}}\ |d\xi|\\
	&\quad \lesssim \int_{\T} \left(\sup_{z_{k}\in \Gamma_{M_+}(\xi)} \left(\int_{D(\tilde{z}_{k},s)} |g'(w)|^{t}\ \frac{dm(w)}{(1-|\tilde{z}_{k}|)^{2}}\right)^{1/t}(1-|z_k|)^{1+\frac{1}{t}-\frac{1}{p}} \right)^{\frac{qs}{q-s}}\ |d\xi|\\
	&\quad \lesssim \int_{\T} \left(\sup_{z_{k}\in \Gamma_{M_+}(\xi)} \left(\int_{D({z_k},4r)} |g'(w)|^{t}\ {dm(w)}\right)^{1/t}(1-|z_k|)^{1-\frac{1}{t}-\frac{1}{p}} \right)^{\frac{qs}{q-s}}\ |d\xi|\\
	&\quad = \int_{\T} \left(\sup_{z_{k}\in \Gamma_{M_+}(\xi)} \frac{\mu_{g}^{1/t}(D(z_k,4r))}{(1-|z_k|)^{\frac{1}{t}+\frac{1}{p}-1}} \right)^{\frac{qs}{q-s}}\ |d\xi|.
	\end{align*}
	Recalling the argument we used in the equivalence of \emph{(2)(a)} and \emph{(3)(a)} we can find a positive constant $N=N(r,\kappa)$ such that 
	\begin{align*}
	&\int_{\T} \left(\sup_{w\in \Gamma_{M}(\xi)} |g'(w)|(1-|w|)^{1+\frac{1}{t}-\frac{1}{p}} \right)^{\frac{qs}{q-s}}\ |d\xi|\\
	&\quad \lesssim N^{\frac{qs(1-t)}{q-s}} \int_{\T} \left(\sup_{z_{k}\in \Gamma_{M_+}(\xi)} \sum_{z_j\in D(z_k,5r)}   \frac{\mu_{g}^{1/t}(D(z_j,r))}{(1-|z_k|)^{\frac{1}{t}+\frac{1}{p}-1}} \right)^{\frac{qs}{q-s}}\ |d\xi|\\
	&\quad \lesssim N^{\frac{qs(1-t)}{q-s}}  N^{\frac{qs}{q-s}}\int_{\T} \left(\sup_{z_{k}\in \Gamma_{M_+}(\xi)}\left(\sup_{z_j\in D(z_k,5r)}   \frac{\mu_{g}^{1/t}(D(z_j,r))}{(1-|z_k|)^{\frac{1}{t}+\frac{1}{p}-1}}\right) \right)^{\frac{qs}{q-s}}\ |d\xi|.
	\end{align*}
	By Lemma~\ref{Wu}, one can take $M_{\ast}>M_{+}>1/2$ such that
	\begin{align*}
	\bigcup_{z_{k}\in \Gamma_{M_{+}}(\xi)} D(z_k,5r)\subset \Gamma_{M_{\ast}}(\xi).
	\end{align*}
	So that, we obtain
	\begin{align*}
	&\int_{\T} \left(\sup_{w\in \Gamma_{M}(\xi)} |g'(w)|(1-|w|)^{1+\frac{1}{t}-\frac{1}{p}} \right)^{\frac{qs}{q-s}}\ |d\xi|\\
	&\quad \lesssim\int_{\T} \left(\sup_{z_{k}\in \Gamma_{M_{\ast}}(\xi)}  \frac{\mu_{g}^{1/t}(D(z_k,r))}{(1-|z_k|)^{\frac{1}{t}+\frac{1}{p}-1}} \right)^{\frac{qs}{q-s}}\ |d\xi| \\
	&\quad \lesssim\int_{\T} \left(\sup_{z_{k}\in \Gamma_{M}(\xi)}  \frac{\mu_{g}^{1/t}(D(z_k,r))}{(1-|z_k|)^{\frac{1}{t}+\frac{1}{p}-1}} \right)^{\frac{qs}{q-s}}\ |d\xi| .
	\end{align*}
	
	Last inequality follows by Lemma~\ref{lemma_sup}. Therefore, we have proved that \emph{(2)(b)} and \emph{(3)(b)} are equivalent.
	\newline

	The equivalence between \emph{(2)(c)} and \emph{(3)(c)} follows  immediately. Observe that for each $D(z_k,r)$
	\begin{align*}
	\frac{\mu_{g}^{1/t}(D(z_k,r))}{(1-|z_k|)^{\frac{1}{p}+\frac{1}{q}+\frac{1}{t}-\frac{1}{s}-1}}\lesssim \sup_{z\in\D} |g'(z)|(1-|z|)^{1+\frac{1}{t}+\frac{1}{s}-\frac{1}{p}-\frac{1}{q}}.
	\end{align*}
	So, one implication holds.
	Now, set $z\in\D$. Applying the mean value property  it occurs that
	$$ |g'(z)|(1-|z|)^{1+\frac{1}{t}+\frac{1}{s}-\frac{1}{p}-\frac{1}{q}}\lesssim \left(\int_{D(z,r)} |g'(w)|^{t}(1-|w|)^{t+\frac{t}{s}-\frac{t}{p}-\frac{t}{q}-1}\ dm(w)\right)^{1/t}.
	$$
	Since the number of the hyperbolic discs $D(z_k,r)$ such that $D(z_k,r)\cap D(z,r)\neq \emptyset$ can be at most a fixed number (see Proposition~\ref{Ncoverdiscs}), we have
	\begin{align*}
	|g'(z)|(1-|z|)^{1+\frac{1}{t}+\frac{1}{s}-\frac{1}{p}-\frac{1}{q}}\lesssim \ \sup_{k}\frac{\mu_{g}^{1/t}(D(z_k,r))}{(1-|z_k|)^{\frac{1}{p}+\frac{1}{q}+\frac{1}{t}-\frac{1}{s}-1}}.
	\end{align*}
	
	Now, we show that \emph{(3)(d)} implies  \emph{(2)(d)}. Fix an arc $I\subset \T$. Then
	\begin{align*}
	\sum_{z_k\in S(I)} & \left(\frac{\mu_{g}^{1/t}(D(z_k,r))}{(1-|z_k|)^{\frac{1}{p}+\frac{1}{t}-1}}\right)^{\frac{pt}{p-t}}(1-|z_k|)
	\\ & \lesssim  \sum_{z_k\in S(I)}  |g'(\tilde{z}_k)|^{\frac{pt}{p-t}}(1-|z_k|)^{2+\frac{pt}{p-t}},
	\end{align*}
	where $|g'(\tilde{z}_k)|:=\sup_{w\in \overline{D(z_k,r)}} |g'(w)|$.
	Using the mean value property on  $D(\tilde{z}_k,s)$ (see Lemma~\ref{MVP}) with $s<3r$, it follows that
	\begin{align*}
	&\sum_{z_k\in S(I)} \left(\frac{\mu_{g}^{1/t}(D(z_k,r))}{(1-|z_k|)^{\frac{1}{p}+\frac{1}{t}-1}}\right)^{\frac{pt}{p-t}}(1-|z_k|)\\
	&\quad\lesssim \sum_{z_k\in S(I)} \frac{(1-|z_k|)^{2+\frac{pt}{p-t}}}{(1-|\tilde{z}_k|)^{2}}\left(\int_{D(\tilde{z}_k,s)} |g'(w)|^{\frac{pt}{p-t}}\ dm(w)\right)
	\end{align*}
	Since $D(\tilde{z}_k,s)\subset D(z_k,4r)$, we obtain
	\begin{align*}
	&\sum_{z_k\in S(I)} \left(\frac{\mu_{g}^{1/t}(D(z_k,r))}{(1-|z_k|)^{\frac{1}{p}+\frac{1}{t}-1}}\right)^{\frac{pt}{p-t}}(1-|z_k|)\\
	&\quad\lesssim  \sum_{z_k\in S(I)} (1-|z_k|)^{\frac{pt}{p-t}} \left(\int_{D({z}_k, 4r)} |g'(w)|^{\frac{pt}{p-t}}\ dm(w) \right)\\
	&\quad\asymp \sum_{z_k\in S(I)} \int_{D({z}_k, 4r)} (1-|w|)^{\frac{pt}{p-t}}|g'(w)|^{\frac{pt}{p-t}}\ {dm(w)} \,.
	\end{align*}
	Taking an arc $I_+\subset \T$ such that $\bigcup_{z_k\in S(I)}D(z_k,4r)\subset S(I_+)$ with $|I|\asymp |I_{+}|$ and employing the finite covering property of the lattice we have 
	\begin{align*}
	&\sum_{z_k\in S(I)} \left(\frac{\mu_{g}^{1/t}(D(z_k,r))}{(1-|z_k|)^{\frac{1}{p}+\frac{1}{t}-1}}\right)^{\frac{pt}{p-t}}(1-|z_k|)
	\lesssim \int_{S(I_+)} (1-|w|)^{\frac{pt}{p-t}}|g'(w)|^{\frac{pt}{p-t}}\ {dm(w)}
	\end{align*}
	which is enough in order to claim that we have accomplished our aim.
	
	Let us continue proving that \emph{(2)(d)} implies \emph{(3)(d)}. Fix an arc $I\subset \T$.  Then
	\begin{align*}
	&\int_{S(I)} (1-|w|)^{\frac{pt}{p-t}}|g'(w)|^{\frac{pt}{p-t}}\ {dm(w)}\\
	& \quad \lesssim \quad \sum_k   \int_{S(I)\cap D(z_k, r)}(1-|w|)^{\frac{pt}{p-t}}|g'(w)|^{\frac{pt}{p-t}}\ {dm(w)}
	\end{align*}
	Bearing in mind that we can take $I_+$ such that $\bigcup_{D(z_k, r)\cap S(I) \neq \emptyset} D(z_k,r)\subset S(I_+)$ with $|I|\asymp |I_{+}|$, we have
	\begin{align*}
	&\int_{S(I)} (1-|w|)^{\frac{pt}{p-t}}|g'(w)|^{\frac{pt}{p-t}}\ {dm(w)}\leq \sum_{z_k\in S(I_{+})} \int_{D(z_k,r)} (1-|w|)^{\frac{pt}{p-t}}|g'(w)|^{\frac{pt}{p-t}}\ {dm(w)}\\
	&\quad \asymp \sum_{z_k\in S(I_{+})} (1-|z_k|)^{\frac{pt}{p-t}}\int_{D(z_k,r)} |g'(w)|^{\frac{pt}{p-t}}\ {dm(w)}.
	\end{align*}
	The application of the submean value property leads to 
	\begin{align*}
	&\int_{S(I)} (1-|w|)^{\frac{pt}{p-t}}|g'(w)|^{\frac{pt}{p-t}}\ {dm(w)}\\
	&\quad \lesssim \sum_{z_k\in S(I_{+})} (1-|z_k|)^{\frac{pt}{p-t}}\int_{D(z_k,r)}
	\left( \frac{1}{(1 - |z_k|)^2}\int_{D(z_k, 2r)}\,|g'(u)|^t\,dm(u) \right)^{\frac{p}{p-t}}\ {dm(w)}\\
	&\quad  \asymp \sum_{z_k\in S(I_{+})} (1-|z_k|)^{\frac{pt}{p-t}-2 \frac{p}{p-t} +2}
	\left( \mu_{g} ^{\frac 1t}(D(z_k,2r)\right)^{\frac{pt}{p-t}}\\
	&\quad  = \sum_{z_k\in S(I_{+})} (1-|z_k|)^{\frac{pt}{p-t}-2 \frac{p}{p-t} +1}
	\left( \mu_{g} ^{\frac 1t}(D(z_k,2r)\right)^{\frac{pt}{p-t}} (1-|z_k|)\\
	& \quad = \sum_{z_k\in S(I_{+})}
	\left( \frac {\mu_{g} ^{\frac 1t}(D(z_k,2r))} {(1-|z_k|)^{\frac{1}{p}+\frac{1}{t}-1}}\right)^{\frac{pt}{p-t}} \,  (1-|z_k|)\,.
	\end{align*}
	Since the hyperbolic discs $D(z_k,2r)$ can be covered at most by a fixed number of hyperbolic discs $D(z_j,r)$ and by the fact that one can take $I_{\ast}\subset \T$ such that $\bigcup_{z_k\in S(I_+)} D(z_k,3r)\subset S(I_\ast)$ with $|I_+|\asymp |I_\ast|$, then
	\begin{align*}
	&\int_{S(I)} (1-|w|)^{\frac{pt}{p-t}}|g'(w)|^{\frac{pt}{p-t}}\ {dm(w)} \lesssim \sum_{z_k\in S(I_{\ast})}
	\left( \frac {\mu_{g} ^{\frac 1t}(D(z_k,r))} {(1-|z_k|)^{\frac{1}{p}+\frac{1}{t}-1}}\right)^{\frac{pt}{p-t}} \,  (1-|z_k|)\,.
	\end{align*}		
	Therefore, we are done.
\end{proof}

Closing this section we consider the action of $T_g: RM(p,q)\rightarrow RM(t,s)$ when one of the parameters is infinite, in particular when $t=+\infty$. In other words, when $T_g: RM(p,q)\rightarrow H^s$, since $RM(\infty,s)=H^s$. The specific case $p=q$, that is when $T_g:A^p\rightarrow H^s$, has already been considered by Wu in \cite{Wu_2011} and subsequently by Miihkinen, Pau, Perälä, and Wang in \cite{MiihkinenPauPeralaWang2020} for the multivariable case. We complete the scene for the more general cases of indices as stated below.


\begin{theorem}\label{RMH}
	Let $1\leq p,q,s<+\infty$. The following statements are equivalent:
	\begin{enumerate}
		\item The operator $T_g: RM(p,q)\rightarrow H^s$ is bounded.
		\item  If $Z=\{z_k\}$ is an $(r,\kappa)$-lattice and denoting $d\mu_g(z)=|g'(z)|^2\ dm(z)$ it holds
		\begin{enumerate}
			\item If $1\leq s<q<+\infty$, $2<p<+\infty$, 
			\begin{align*}
			\int_{\T} \left(\sum_{z_k\in \Gamma(\xi)} \left(\frac{\mu_{g}^{1/2}(D(z_k,r))}{(1-|z_k|)^{1/p}}\right)^{\frac{2p}{p-2}} \right)^{\frac{(p-2)qs}{(q-s)2p}}\ |d\xi|<+\infty.
			\end{align*}
			\item If $1\leq s< q<+\infty$, $1\leq p\leq 2$, 
			\begin{align*}
			\int_{\T} \left(\sup_{z_k\in \Gamma(\xi)} \frac{\mu_{g}^{1/2}(D(z_k,r))}{(1-|z_k|)^{1/p}} \right)^{\frac{qs}{q-s}}\ |d\xi|<+\infty.
			\end{align*}
			\item If $1\leq q< s<+\infty$, $1\leq p<\infty$ or $1\leq q=s<+\infty$, $1\leq p\leq 2$, 
			\begin{align*}
			\sup_{k} \frac{\mu_{g}^{1/2}(D(z_k,r))}{(1-|z_k|)^{\frac{1}{p}+\frac{1}{q}-\frac{1}{s}}}<+\infty.
			\end{align*}
			\item If $1\leq q=s<+\infty$, $2< p<+\infty$, 
			\begin{align*}
			\sup_{\xi\in\T} \left(\sup_{\xi\in I}  \sum_{z_{k}\in S(I)} \left(\frac{\mu_{g}^{1/2}(D(z_k,r))}{(1-|z_k|)^{1/p}}\right)^{\frac{2p}{p-2}} (1-|z_k|) \right)^{\frac{p-2}{2p}}<+\infty.
			\end{align*}
		\end{enumerate}
		\item The function $g\in \mathcal{H}(\D)$ satisfies that
		\begin{enumerate}
			\item If $1\leq s<q<+\infty$, $2<p<+\infty$, 
			\begin{align*}
			|g'(z)|(1-|z|)^{\frac{1}{2}}\in T_{\frac{2p}{p-2}}^{\frac{qs}{q-s}}(1).
			\end{align*}
			\item If $1\leq s< q<+\infty$, $1\leq p\leq 2$, 
			\begin{align*}
			|g'(z)|(1-|z|)^{1-\frac{1}{p}}\in T_{\infty}^{\frac{qs}{q-s}}.
			\end{align*}
			\item If $1\leq q< s<+\infty$, $1\leq p<\infty$ or $1\leq q=s<+\infty$, $1\leq p\leq 2$, $g\in\mathcal{B}^{1+\frac{1}{s}-\frac{1}{p}-\frac{1}{q}}$.
			\item If $1\leq q=s<+\infty$, $2< p<+\infty$, the measure $|g'(z)|^{\frac{2p}{p-2}}(1-|z|)^{\frac{p}{p-2}}\ dm(z)$ is a Carleson measure.
		\end{enumerate}
	\end{enumerate}	
\end{theorem}
\begin{proof}
	The proof of the equivalences between \emph{(1)} and \emph{(2)} follows as a consequence of Theorem~\ref{theoremCarlesonTent}.
	 This is due to the equivalent description of the Hardy norm as  
	\begin{align*}
	\|T_g(f)\|_{H^s}\asymp \left(\int_{\T} \left(\int_{\Gamma(\xi)} |f(w)|^2 |g'(w)|^{2}\ dm(z)\right)^{s/2}\ |d\xi|\right)^{1/s}
	\end{align*}
	(see \cite[Theorem 1.3, p. 172]{Pavlovic}) and to the consideration  $d\mu(z)=|g'(z)|^2\ dm(z)$.
	
	The remaining equivalence results using the same arguments that were  used in the proof of Theorem~\ref{RMRM}.
	
\end{proof}

\noindent {\bf Acknowledgments.} We are grateful to Professor Manuel D. Contreras and Professor Luis Rodríguez-Piazza for their useful comments.


\begin{thebibliography}{99}
	
	
	\bibitem{Aguilar-Contreras-Piazza} T. Aguilar-Hern\'andez, M.D. Contreras, and L. Rodr\'\i guez-Piazza,  {\sl Average radial integrability spaces of analytic functions},
	J. Funct. Anal. \textbf{282}, 1 (2022), Article number: 109262.
	
	\bibitem{Aguilar-Contreras-Piazza_2}T. Aguilar-Hern\'andez, M. D. Contreras, and L. Rodr\'\i guez-Piazza, {\sl Integration operators in average radial integrability spaces of analytic functions,}  Mediterr. J. Math. \textbf{18} (2021), 117.
	
	\bibitem{AS} A. Aleman, A.G. Siskakis, {\sl An integral operator on $H^p$}, Complex Variables
	Theory Appl., {\bf28} (1995), 149-158.
	
	\bibitem{AS2} A. Aleman, A.G. Siskakis, {\sl Integration operators on Bergman spaces}, Indiana University Mathematics Journal, {\bf46} (1997), 337--356.
	
	\bibitem{Arsenovic} M. Arsenović, {\sl  Embedding derivatives of $\mathcal{M}$-Harmonic Functions into $L^p$ Spaces}, Rocky Mountain J. Math., {\bf 29} (1999), 149-158.
	
	
    \bibitem{Bagby} R. J. Bagby, {\sl An extended inequality for the maximal function}, Proc. Amer. Math. Soc., \textbf{48}
    (1975), 419-422.
    
    \bibitem{Carleson_1958} L. Carleson, {\sl An interpolation problem for bounded analytic functions}, Amer. J. Math., \textbf{80} (1958) 921-930.
    
    
    \bibitem{Carleson_1962} L. Carleson, {\sl Interpolations by bounded analytic functions and the corona problem}, Ann. of Math., \textbf{76} (1962),547-559.
    
	
	\bibitem{Cohn_Verbitsky_200}  W. S. Cohn and I. E. Verbitsky, {\sl Factorization of tent spaces and Hankel operator}, J. Funct. Anal., {\bf 175} (2000) 308-329.
	
	\bibitem{CMS} R. R. Coifman, Y. Meyer, E. M. Stein,{\sl  Some new function spaces and their applications to harmonic analysis},
	J. Funct. Anal., {\bf 62} (1985), 304-335.
	  
	  \bibitem{Du} P. Duren, {\sl Theory of $H^p$ spaces}, Dover, New York, 2000.
	  
	  
	\bibitem{duren_schuster_2004} P. Duren and A. Schuster, {\sl Bergman Spaces}, American Mathematical Society, Providence, RI, 2004.
	
	 \bibitem{G} J.B.Garnett, {\sl Bounded Analytic Functions}, Graduate Texts in Mathematics, Springer-Verlag, 2007.
	 
	\bibitem{HKZ}H. Hedenmalm, B. Korenblum, K.Zhu, {\sl Theory of Bergman spaces},  Graduate Texts in Mathematics, 	Springer-Verlag, 2000.
	
%

    
	
	\bibitem{Jevtic_1996} M. Jevtić, {\sl  Embedding derivatives of $\mathcal{M}$-harmonic Hardy spaces $H^p$ into Lebesgue spaces, $0<p<2$}, Rocky Mountain J. Math., {\bf 26} (1996), 175-187.
	
	\bibitem{Kalton} N. Kalton, {\sl Convexity, type and the three space problem}, Studia Math., {\bf 69} (1981), 247--287.
		
	\bibitem{Luecking} D. H. Luecking, {\sl Embedding Theorem for Spaces of Analytic Functions via Khinchine's Inequality,} Michigan Math. J., {\bf 40} (1993), 333–358.
	
	\bibitem{Luecking1987} D. H. Luecking, {\sl Trace Ideal Criteria for Toeplitz Operators,} J. Funct. Anal.,  {\bf 73} (1987), 345-368.
	
	\bibitem{MiihkinenPauPeralaWang2020} S. Miihkinen, J. Pau, A. Perälä, and M. Wang, {\sl Volterra type integration operators from Bergman spaces to Hardy spaces,}  J. Funct. Anal., {\bf 279} (2020).
	
	\bibitem{Pau_Pelaez_2010} J. Pau and J. A. Peláez, {\sl Embedding theorems and integration operators on Bergman spaces with rapidly decreasing weights}, J. Funct. Anal., {\bf 259}, 10 (2010), 2727–2756.
	
	\bibitem{Pau_Pelaez_2012} J. Pau and J. A. Peláez, {\sl Volterra type operators on Bergman spaces with exponential weights}, Contemporary Mathematics, {\bf 561} (2010), 239–252. 
	
	\bibitem{Pau_2016} J. Pau, {\sl Integration operators between Hardy spaces on the unit ball of $\C^n$,} J. Funct. Anal., {\bf 270} (2016), 134–176.
	
	\bibitem{Pavlovic} M. Pavlovi\'c, {\sl On the Littlewood-Paley $g$-function and Calder\'on's area theorem,} Expo. Math., {\bf 31} (2013), 169-195. 
	
	\bibitem{Pelaez_Rattya_Sierra_2015} J. A. Peláez, J. Rättyä, and K. Sierra, {\sl Embedding Bergman spaces into tent spaces}, Mathematische Zeitschrift, {\bf 281} (2015), 1215-1237. 
	
	\bibitem{Perala_2018} A. Perälä, {\sl Duality of holomorphic Hardy type tent spaces}, (2018). ArXiv:1803.10584, 2018.
	
	\bibitem{Pom} Ch. Pommerenke, {\sl Schlichte Funktionen und Funktionen von beschrankter mittler Oszilation}, Comment. Math. Helv., {\bf 52} (1977), 122–129
	
	\bibitem{Ronkin_1974} L. I. Ronkin, {\sl Introduction to the Theory of Entire Functions of Several Variables}, American Mathematical Society, Providence, RI, 1974.

	\bibitem{rudin_real_1987} W. Rudin, {\sl Real and complex analysis}, McGraw-Hill International Ed. 3rd ed., New York, 1987.

\bibitem{Wu_2006} Z. Wu, {\sl Area operator, on Bergman spaces}, Sci. China Math., {\bf 49} (2006), 987-1008.
	
	\bibitem{Wu_2011} Z. Wu, {\sl Volterra operator, area integral and Carleson measures}, Sci. China Math., {\bf 54} (2011), 2487-2500.
	
		\bibitem{Z} A. Zygmund, {\sl Trigonometric series}, Cambridge University Press, London, 1987.
\end{thebibliography}
\end{document}